\documentclass[a4paper,10pt]{amsart}
\usepackage{amsmath,amssymb,amsthm,mathtools}
\usepackage{epsfig}
\usepackage{graphicx}

\makeatletter
\def\Ginclude@eps#1{%
 \message{<#1>}%
  \bgroup
  \def\@tempa{!}%
  \dimen@\Gin@req@width
  \dimen@ii.1bp%
  \divide\dimen@\dimen@ii
  \@tempdima\Gin@req@height
  \divide\@tempdima\dimen@ii
    \includegraphics{#1}%
  \egroup}
\makeatother

\newcommand{\CM}{\overline{\operatorname{\mathcal{M}}}}

\newcommand{\IM}{\operatorname{\mathcal{M}}}

\newcommand{\LL}{\operatorname{\mathcal{L}}}

\newcommand{\PP}{\operatorname{\mathfrak{P}}}

\newcommand{\Si}{\dot{S}}

\newcommand{\CR}{\bar{\partial}}

\newcommand{\Mph}{M_{\phi}}

\newcommand{\Aut}{\operatorname{Aut}}
\newcommand{\Symp}{\operatorname{Symp}}

\newcommand{\ev}{\operatorname{ev}}

\newcommand{\CZ}{\operatorname{CZ}}

\newcommand{\del}{\partial}

\newcommand{\Ju}{\underline{J}}

\newcommand{\IR}{\operatorname{\mathbb{R}}}
\newcommand{\IN}{\operatorname{\mathbb{N}}}
\newcommand{\ID}{\operatorname{\mathbf{D}}}

\newcommand{\Ih}{\operatorname{\mathbf{h}}}

\newcommand{\IP}{\operatorname{\mathbb{P}}}

\newtheorem{theorem}{Theorem}[section]

\newtheorem{definition}[theorem]{Definition}

\newtheorem{corollary}[theorem]{Corollary}
\newtheorem{rem}[theorem]{Remark}
\newtheorem{ex}[theorem]{Example}
\newenvironment{remark}{\begin{rem}\rm}{\qee\end{rem}}
\newenvironment{example}{\begin{ex}\rm}{\qee\end{ex}}
\newcommand{\qee}{\mbox{\hspace{0.2mm}}\hfill$\triangle$}

\title{Nijenhuis operator in contact homology and descendant recursion in symplectic field theory}
\author{Paolo Rossi}
\begin{document}

\begin{abstract}
In this paper we investigate the algebraic structure related to a new type of correlator associated to the moduli spaces of $S^1$-parametrized curves in contact homology and rational symplectic field theory. Such correlators are the natural generalization of the non-equivariant linearized contact homology differential (after Bourgeois-Oancea) and give rise to an invariant Nijenhuis (or hereditary) operator (\`a la Magri-Fuchssteiner) in contact homology which recovers the descendant theory from the primaries. We also sketch how such structure generalizes to the full SFT Poisson homology algebra to a (graded symmetric) bivector. The descendant hamiltonians satisfy to recursion relations, analogous to bihamiltonian recursion, with respect to the pair formed by the natural Poisson structure in SFT and such bivector. In case the target manifold is the product stable Hamiltonian structure $S^1\times M$, with $M$ a symplectic manifold,  the recursion coincides with genus $0$ topological recursion relations in the Gromov-Witten theory of $M$.
\end{abstract}

\maketitle

\tableofcontents

\markboth{P. Rossi}{Nijenhuis operator in contact homology}

\section*{Introduction}

Starting from the early nineties, Boris Dubrovin and many of his collaborators have studied the relation between Gromov-Witten theory and the theory of integrable systems of PDEs, which was first noticed by Witten in \cite{W}. The axioms of Frobenius manifold, cf. \cite{D}, encode all the properties that are satisfied by the algebraic structure generated by rational Gromov-Witten theory and that, in particular, generate an integrable Hamiltonian system of evolutionary PDEs. The structure of Frobenius manifold has proven to be central in many different areas of mathematics, from algebra to singularity theory, and provides for instance the most immediate approach to mirror symmetry.\\

One of the consequences of the axioms of a (homogenous) Frobenius manifold is that the associated Hamiltonian system of PDEs is actually bihamiltonian. Bihamiltonian structures where introduced by Magri in \cite{M1} in the analysis of the so-called Lenard scheme
(see e.g \cite{GGKM}) to construct the KdV integrals. They consist of a manifold endowed with two Poisson tensors $\Pi_1$ and $\Pi_2$, mutually compatible in the sense that their Schouten-Nijenhuis bracket $[\Pi_1,\Pi_2]$ vanishes. Under the condition that the Poisson pencil $\Pi_\lambda=\Pi_2 - \lambda \Pi_1$ (a one-parametric family of Poisson tensors) has constant co-rank there exists a simple recursive procedure for constructing a sequence of commuting integrals for both Poisson structures (see also \cite{DZ} and the author's survey \cite{R2}).

\begin{theorem}[\cite{M1}]
Let P be a manifold endowed with compatible Poisson tensors $\Pi_1$, $\Pi_2$ and associated Poisson brackets $\{\, \cdot\, ,\, \cdot\, \}_{1}$, $\{\, \cdot\, ,\, \cdot\, \}_{2}$. Let $k=\mathrm{corank}\, \Pi_1=\mathrm{corank}\, (\Pi_1 +\epsilon \Pi_2)$ for arbitrary sufficiently small $\epsilon$. Then the coefficients of the Taylor expansion
$$c^\alpha(x,\lambda)=c^\alpha_{-1}(x) + \frac{c^\alpha_0(x)}{\lambda} + \frac{c^\alpha_{1}(x)}{\lambda^2} + \ldots$$
of the Casimirs $c^\alpha(x,\lambda)$, $\alpha=1,\ldots,k$ of the Poisson tensor $\Pi_\lambda=\Pi_2 - \lambda \Pi_1$ commute with respect to both Poisson brackets,
$$\{c^\alpha_i,c^\beta_j\}_{1,2}=0\ ,\ i,j=-1,0,1,\ldots.$$
Moreover $\{\, \cdot\, ,c^\alpha_{i+1}\}_1=\{\, \cdot\, ,c^\alpha_{i}\}_2$, $i=-1,0,1,\ldots$.\\
\end{theorem}

While in the literature this procedure is often called bihamiltonian recursion, we will use the term {\it bihamiltonian reconstruction}, to avoid confusion with the property of a sequence of symmetries $\{I_{\alpha,i}\}_{\alpha=1,\ldots,k;i=-1,0,1,\ldots}$ of being in {\it bihamiltonian recursion} if
$$\{\, \cdot\, ,I_{\alpha,i}\}_2=\sum_{\substack{\beta=1,\ldots,k \\j=0,\ldots,i+1}}R^{\beta,j}_{\alpha,i}\{\, \cdot\, ,I_{\beta,j}\}_1$$
for some constant coefficients $R^{\beta,j}_{\alpha,i}$, $\alpha,\beta=1,\ldots,k$, $i,j=-1,0,1,\ldots$.\\

In the context of Gromov-Witten theory and the associated Frobenius manifolds such reconstruction can be used to recover at least part of the symmetries for the Hamiltonian system. Actually, with great generality, given a homogeneous Frobenius manifold, finding a fundamental solution for the so called deformed flat connection (a special case of topological recursion relations for rational one-descendant GW invariants, see \cite{DZ}) is a more powerful way to reconstruct the algebra of symmetries and, moreover, with this method the solution is automatically normalized to match the generating series of rational one-descendant GW invariants (often called $J$-function in the Gromov-Witten literature). Besides, the symmetries found this way are in bihamiltonian recursion anyway.\\

In contrast, bihamiltonian reconstruction is not, in general, directly related with enumerative geometry, which results in discrepancies with the $J$-function. Worse, even when the hypothesis of the above theorem are satisfied, it might happen that the two Poisson tensors have some Casimir in common. This means that, when starting with such common Casimirs as $c^\alpha_{-1}(x)$, in the above Taylor expansion, all of the other coefficients trivially vanish. This is precisely what happens in the case of the Gromov-Witten theory of the projective line $\IP^1$, where bihamiltonian recursion is only capable of recovering the symmetries generated by descendants of the K\"ahler class, but not those of the unity class. This pathology of the Poisson pencil is called {\it resonance}.\\

This paper deals with symplectic field theory (SFT) and a recursion procedure for descendants that has much in common with bihamiltonian recursion in Gromov-Witten theory but is instead actually a generalization of the method of deformed flat coordinates. Introduced by H. Hofer, A. Givental and Y. Eliashberg in 2000 \cite{EGH}, SFT is a very large
project and can be viewed as a topological quantum field theory approach to Gromov-Witten theory. Beside providing a
unified view on established pseudoholomorphic curve theories like symplectic Floer homology, contact homology and
Gromov-Witten theory, it sheds considerable light on the appearence of infinite dimensional Hamiltonian systems in the theory of holomorphic curves (see \cite{R2} for a review on this topic which includes SFT).\\

Indeed, symplectic field theory leads to algebraic invariants with very rich algebraic structures and in particular, as it was pointed out by Eliashberg in his ICM 2006 plenary talk (\cite{E}), the integrable systems of rational Gromov-Witten theory very naturally appear in rational symplectic field theory by using the link between the rational symplectic field theory of prequantization spaces in the Morse-Bott version and the rational Gromov-Witten potential of the underlying symplectic manifold (see the recent papers \cite{R1}, \cite{R2}). After introducing gravitational descendants (see \cite{F2}) along the lines of Gromov-Witten theory, it is precisely the natural algebraic structure of SFT that provides a natural link between holomorphic curves and (quantum) integrable systems.\\

In this paper we explore the potentiality of an intrinsic difference between Gromov-Witten and symplectic field theory: the moduli spaces of holomorphic maps studied by the latter carry special evaluation maps controlling the relative gluing angle of different components of a multi-floor configuration (the SFT generalization of nodal curves, see \cite{EGH}). These can be used to define new correlators that were not present in the original theory, but give interesting recursive formulas for one-point descendants (and probably beyond), which are similar but not equivalent to bihamiltonian reconstruction, topological recursion and, ultimately, to the integrability properties of the SFT infinite dimensional hamiltonian system.\\

In many senses these extra correlators which control the gluing angles of different components of curves in the boundary strata are a natural generalization of the non-$S^1$-equivariant differential for the non-equivariant linearized contact homology of Bourgeois-Oancea, \cite{BO} (see also \cite{FR1}), to non-linearized contact homology and full rational SFT. With this in mind we conclude that the non-equivariant differential plays a role similar to the second Poisson structure of Gromov-Witten theory with respect to gravitational descendants. More precisely we show how our generalized non-equivariant differential (the potential encoding such new correlators), denoted by $N$, satisfies (in contact homology) the axioms of a Nijenhuis operator.\\

Recall from Magri and Fuchssteiner, \cite{M2}\cite{Fu}, that on a manifold $M$, a Nijenhuis (or hereditary) operator $N \in \mathcal{T}^{(1,1)}M$ (where $\mathcal{T}^{(k,l)}M$ denotes the space of $(k,l)$-tensor fields on $M$) is one whose Nijenhuis torsion $T(N)\in \mathcal{T}^{(2,1)}M$:
\begin{equation*}
\begin{split}
T(N)(X,Y):&=[NX,NY]-N([NX,Y]+[X,NY])+N^2[X,Y]=\\
&=(\mathcal{L}_{NX}N)Y-N(\mathcal{L}_X N) Y
\end{split}
\end{equation*}
vanishes (here $X$ and $Y$ are any two graded vector fields). In components the Nijenhuis torsion reads:
\begin{equation}\label{torsion}
T(N)^a_{cd}=N^{a}_b\left(\frac{\del N^b_{c}}{\del x^{d}}-\frac{\del N^b_{d}}{\del x^c}\right)-\frac{\del N^{a}_{c}}{\del x^{b}} N^{b}_{d}+\frac{\del N^{a}_{d}}{\del x^{b}} N^{b}_{c}
\end{equation}
This condition ensures that, given a sequence of commuting vector fields
$$X_{\alpha,0}\in\mathcal{T}^{(1,0)}M,\qquad [X_{\alpha,0},X_{\beta,0}]=0,\qquad \alpha,\beta=1,\ldots,n$$ that are also symmetries of the operator $N$, i.e. $\mathcal{L}_{X_{\alpha,0}}N=0$, we can enlarge the commuting system by recursively applying the operator $N$:
$$X_{\alpha,k}:=N^k(X_{\alpha,0}),\qquad [X_{\alpha,k},X_{\beta,j}]=0,\qquad k,j\in\mathbb{N}.$$
Given a Poisson tensor $\Pi\in \mathcal{T}^{(2,0)}M$ on $M$, in \cite{MM} Magri and Morosi further studied the compatibility conditions of a Nijenhuis operator $N$ with the Poisson structure. $N$ is said to be compatible with $\Pi$ if the following two conditions hold:
$$N\circ \Pi=\Pi \circ {}^t N$$
$$\Pi^{lj}\left(\frac{\del N^k_m}{\del x^l}-\frac{\del N^k_l}{\del x^m}\right)-\Pi^{kl}\frac{\del N^j_m}{\del x^l} - N^l_m \frac{\del \Pi^{kj}}{\del x^l} + N^j_l \frac{\del \Pi^{kl}}{\del x^m} = 0$$
When these equations are satisfied, the pair $(\Pi,N)$ is called a Poisson-Nijenhuis structure on $M$. The main property, then, is that one can define a sequence of $(2,0)$-tensors $\Pi_k=N^k \circ \Pi$, $k=0,1,2,\ldots$ which are Poisson and are pairwise compatible in the sense that they pairwise form Poisson pencils. In particular we get the bihamiltonian structure $(\Pi_0=\Pi, \Pi_1=N \circ \Pi)$.\\

The theory of Poisson-Nijenhuis structures is very well developed and its relation with integrability is deep. In particular the study of the spectrum of $N$ plays a fundamental role, in that the eigenvalues of $N$ form a system of commuting symmetries in bihamiltonian recursion (see for instance \cite{MM},\cite{KSM},\cite{DLF}).\\

The first part of the paper deals with contact homology of contact manifolds. In this case, thanks to the total absence of non-constant nodal configurations (due to the maximum principle for holomorphic curves in symplectizations), we prove that the knowledge of $N$ and of the primary theory (contact homology differential $X$, with no descendants) is sufficient to completely reconstruct the descendant vector fields $X_{\alpha,n}$ as differential operators in the variables associated to Reeb orbits on contact homology.\\

In the case of full rational SFT our bihamiltonian recursion is slightly less effective because of the presence of non-constant nodal curves. Formally the result is similar, i.e. the descendant Hamiltonians $\Ih_{\alpha,n}$  satisfy recursion relations which are completely analogous to bihamiltonian recursion for a pair of bivectors $\Pi,\omega$, where $\Pi$ is the natural Poisson structure on the SFT homology algebra and $\omega$ is a \emph{graded symmetric} even bivector which is the SFT-generalization of $N$. \\

In any case this looks like a fundamental step in understanding the relation between completeness of the contact homology vector field system, or the SFT Hamiltonian system, and the underlying symplectic topology of the target manifold. Indeed the full information is contained in the non-equivariant correlators forming $N$ and $\omega$, and we plan to study the consequences in a subsequent publication.\\

Finally, this paper aims more to convey the basic ideas of the constructions and proofs involved in our results than to give a fully rigorous exposition. A brief discussion on the level of rigour at which our arguments are presented can be found in Remark \ref{rigour}.\\

\noindent{\bf Acknowledgements.}\\
Part of the work was conducted while I was a postdoc of Fondation de Sciences Mathematiques de Paris at the Institut de Mathematiques de Jussieu, Paris VI and part during my postdoc at the Institute for Mathematics of the University of Zurich (partially supported by SNF Grant No. PDAMP2\_137151/1). I would like to thank Y. Eliashberg, O. Fabert and D. Zvonkine for useful discussions. The final part of the work was actually completed during my visit at Stanford University at the beginning of 2012. I am grateful to Yakov Eliashberg for the invitation. Finally I would like to thank the American Institute of Mathematics for allowing me to organize the workshop on \emph{Integrable systems in Gromov-Witten and symplectic field theory}, in January 2012, during which I could speak to Boris Dubrovin, who then noticed how my equations for $N$ were equivalent to vanishing of Nijenhuis torsion and who pointed me to the correct reference (Magri, Morosi and Fuchssteiner) to see the big picture.\\

\vspace{0.5cm}

\section{Notions from Symplectic Field Theory}
Symplectic field theory (SFT), introduced by Y. Eliashberg,
A. Givental and H. Hofer in \cite{EGH}, consists in a unified and comprehensive approach to
the theory of holomorphic curves in symplectic and contact topology. In the spirit of a topological field theory, it assigns algebraic invariants
to closed manifolds with a stable Hamiltonian structure. We recall here the main ideas from \cite{EGH},\cite{FR},\cite{FR1}.\\

\subsection{Stable Hamiltonian structures and contact structures}
A Hamiltonian structure (see \cite{BEHWZ}) on a closed $(2n-1)$-dimensional manifold $V$ is a closed two-form $\Omega$
on $V$ of maximal rank $2n-2$. This means that $\ker\Omega=\{v\in TV:\Omega(v,\cdot)=0\}$
is a one-dimensional distribution. A Hamiltonian structure is called stable if there exists a
one-form $\lambda$ and a vector field $R$ (called Reeb vector field) on $V$ such that $R$ generates $\ker\Omega$, $\lambda(R)=1$ and $\iota_R d\lambda=0$. Notice that, this way, $\xi=\ker\lambda$ is a symplectic hyperplane distribution. Also, notice that if $R$ exists, it is completely determined by $\Omega$ and $\lambda$.\\

\begin{example}\label{contact}
Any contact form $\lambda$ on $V$ provides a stable Hamiltonian structure $(\Omega:=d\lambda,\lambda)$ on $V$ where the symplectic hyperplane distribution coincides with the contact structure.\\
\end{example}
\begin{example}\label{S1bundle}
Given a principal circle bundle $\pi: V\to M$ over a closed symplectic manifold $(M,\Omega_M)$ and any connection $1$-form $\lambda$, $(\Omega=\pi^*\Omega_M,\lambda)$ is a stable Hamiltonian structure on $V$.
\end{example}
\begin{example}
Given a closed symplectic manifold $(M,\Omega_M)$ and a symplectomorphism $\phi\in\Symp(M,\Omega_M)$, consider the symplectic mapping torus $V=\Mph=\IR\times M/\{(t,p)\sim(t+1,\phi(p))\}$. The natural splitting $TV=TS^1\oplus TM$ allows to define the lift $\Omega$ of $\Omega_M$ to $V$. Then $(\Omega,\lambda=dt)$, where $t$ is the natural $S^1$-coordinate, is a stable Hamiltonian structure (with integrable symplectic distribution $\ker\lambda$).
\end{example}

Given a stable Hamiltonian structure $(V,\Omega,\lambda)$, consider a complex structure $\Ju_\xi$ on the hyperplane distribution $\xi=\ker\lambda$ which is $\Omega|_{\xi}$-compatible (i.e. $\Omega|_{\xi}(\cdot,\Ju_{\xi}\cdot)$ is a metric on $\xi$). Such complex structures form a non-empty, contractible set. We extend $\Ju_\xi$ uniquely from $\xi$ to an almost complex structure $\Ju$ on the cylinder $\IR\times V$ by requiring that $\Ju$ is $\IR$-invariant and $\Ju\del_s=R$, $\del_s$ being the $\IR$-direction.\\

\subsection{Symplectic field theory}
Symplectic field theory assigns algebraic invariants to closed manifolds $V$ with a stable Hamiltonian structure. In this exposition we will however restrict to the special cases of contact manifolds (in the rest of the paper we will sometimes consider some particularly well behaved stable Hamiltonian structures too, see the rest of the section and, in particular, remark \ref{noncontact}, for more details). The invariants are defined by counting $\Ju$-holomorphic curves in $\IR\times V$ with finite energy.\\

Let us recall the definition of moduli spaces of holomorphic curves studied in rational SFT of contact manifolds. Let $\Gamma^+,\Gamma^-$ be two ordered sets of closed orbits $\gamma$ of the Reeb vector field
$R$ on $V$, i.e., $\gamma: \IR\to V$, $\gamma(t+T)=\gamma(t)$, $\dot{\gamma}=R$, where $T>0$ denotes the period of $\gamma$. Here we assume that the contact structure is nondegenerate, i.e. all closed orbits of the
Reeb vector field are nondegenerate in the sense of \cite{BEHWZ}; in particular, the set
of closed Reeb orbits is discrete. Given a closed Reeb orbit $\gamma$ of any multiplicity, we will denote by $\bar\gamma$ its underlying simple Reeb orbit. At each simple Reeb orbit we will fix a closed form $d\phi_{\bar\gamma}$ generating $H^1(S^1)$.\\

Then the (parametrized) moduli space $\IM^0_{r,A}(\Gamma^+,\Gamma^-)$ consists of tuples $(u,(z_k^{\pm}),(z_i))$, where $(z^{\pm}_k),(z_i)$ are three disjoint ordered sets of points on $\IP^1$, which are called positive and negative punctures, and additional marked points, respectively. We will also fix an asymptotic marker at each puncture, i.e. a ray originating at the puncture. The map $u: \Si \to \IR \times V$  starting from the punctured Riemann surface $\Si = \IP^1 \setminus  \{(z_k^{\pm})\}$ is required to satisfy the Cauchy-Riemann equation
\begin{equation*}
 \CR_{\Ju} u = du + \Ju(u) \cdot du \cdot i = 0
\end{equation*}
with respect to the complex structure $i$ on $\IP^1$ and an $\IR$-invariant almost complex structure $J$ on $V\times \IR$ compatible with the contact structure. Assuming we have chosen cylindrical coordinates $\psi^{\pm}_k: \IR^{\pm}
\times S^1 \to \Si$ around each puncture $z^{\pm}_k$ in the sense that $\psi_k^{\pm}(\pm\infty,t)=z_k^{\pm}$,
the map $u$ is additionally required to show for all $k=1,...,n^{\pm}$ the
asymptotic behaviour
\begin{equation*}
 \lim_{s\to\pm\infty} (u \circ \psi^{\pm}_k) (s,t+t_0) =
 (\pm \infty,\gamma^{\pm}_k(T^{\pm}_kt))
\end{equation*}
with $t_0\in S^1$  and with orbits $\gamma^{\pm}_k\in \Gamma^{\pm}$, where $T^{\pm}_k>0$ denotes period of $\gamma^{\pm}_k$. In order to assign an absolute homology class $A$ to a holomorphic curve $u:\Si\to\IR\times V$, let us assume for simplicity that $H_1(V)$ has no torsion so that we can employ spanning surfaces $u_{\gamma}$ connecting a given closed Reeb orbit $\gamma$ in $V$ to a linear combination of
circles $c_s$ representing a basis of $H_1(V)$,
\[\del u_{\gamma} = \gamma - \sum_s n_s\cdot c_s \]
in order to define
\[ A = [u_{\Gamma^+}] + [u(\Si)] - [u_{\Gamma^-}], \]
where $[u_{\Gamma^{\pm}}] = \sum_{n=1}^{s^{\pm}} [u_{\gamma^{\pm}_n}]$ viewed as singular chains. \\

Observe that the group $\Aut(\IP^1)$ of M\"obius transformations acts on elements in $\IM^0=\IM^0_{r,A}(\Gamma^+,\Gamma^-)$ in an obvious way,
\begin{equation*} \varphi.(u,(z^{\pm}_k),(z_i)) = (u\circ\varphi^{-1},(\varphi(z^{\pm}_k)),(\varphi(z_i))),\;\;\;\varphi\in\Aut(\IP^1), \end{equation*}
and we obtain the moduli space $\IM=\IM_{r,A}(\Gamma^+,\Gamma^-)$ studied in symplectic field theory by dividing out this action and the natural $\IR$-action on the target manifold $(\IR\times V,\Ju)$. Furthermore it was shown in \cite{BEHWZ} that this moduli space can be compactified to a moduli space $\CM=\CM_{r,A}(\Gamma^+,\Gamma^-)$ by adding moduli spaces of multi-floor curves with nodes. In particular the appropriate transversality and gluing theorems are conjectured to give that the moduli space has a codimension-one boundary components $\del\CM$ which are $k_\gamma$ quotients of products $\CM_1\times\CM_2= \CM_{r_1,A_1}(\Gamma_1^+,\Gamma_1^-\cup\{\gamma\})\times\CM_{r_2,A_2}(\{\gamma\}\cup\Gamma_2^+,\Gamma_2^-)$ of lower-dimensional moduli spaces by the $T^2$-action which rotates the asymptotic markers at the connecting orbit $\gamma$. Indeed every such quotient appears with a factor $\kappa_\gamma$ in the boundary: here the idea is that a $1$-floor rational curve can degenerate to a $2$-floor rational curve where the datum of asymptotic markers at the connecting node/Reeb orbit $\gamma$ is missing, hence the $T^2$-quotient, but there are $\kappa_\gamma$ ways to smooth out the node of a curve which is asymptotic to a Reeb orbit of multiplicity $\kappa_\gamma$. Another way to see the $T^2$ quotients with factor $\kappa_\gamma$ above is the result of first considering two-floor curves with matching asymptotic markers at the connecting Reeb orbit, whose moduli space is described as a fiber product over evaluation maps at the markers from the moduli space to the simple Reeb orbit $\bar\gamma$ underlying the connecting orbit $\gamma$ (preimage of the diagonal in $\bar\gamma\times\bar\gamma$), and then quotienting out the remaining $S^1$-action of (simultaneously) rotating the markers (this way the $\kappa_\gamma$ factor is seen as coming from the degree of the evaluation map to the underlying simple Reeb orbit, see below).\\

Let us now briefly introduce the algebraic formalism of rational SFT as described in \cite{EGH}: \\

Let us fix a trivialization of the symplectic bundle $(\xi, \Omega|_\xi)$ over each curve $c_i$. This induces a homotopically unique trivialization of the same bundle over each periodic Reeb orbit $\gamma$ via the spanning surface $u_\gamma$. Let us use this trivialization to define the Conley-Zehnder index of the Reeb orbit (the Maslov index of the path in $Sp(2m-2,\IR)$ given by the linearized Reeb flow along $\gamma$). Recall that a multiply-covered Reeb orbit $\gamma=\bar\gamma^k$ is called bad if
$\CZ(\gamma)\neq\CZ(\bar\gamma)\mod 2$, where $\CZ(\gamma)$ denotes the
Conley-Zehnder index of $\gamma$. Calling a Reeb orbit $\gamma$ {\it good} if it is not bad, denote by $\mathcal{P}$ the space of good Reeb orbits. We assign to every
good Reeb orbit $\gamma$ two formal graded variables $p_{\gamma},q_{\gamma}$ with grading
\begin{equation*}
|p_{\gamma}|=m-3-\CZ(\gamma),|q_{\gamma}|=m-3+\CZ(\gamma)
\end{equation*}
when $\dim V = 2m-1$.\\

Assuming we have chosen a basis $A_0,\ldots,A_M$ of $H_2(V)$, we assign to every $A_i$ a formal
variable $z_i$ with grading $|z_i|=- 2 c_1(A_i)$. In order to include higher-dimensional moduli spaces we further assume that a string
of closed (homogeneous) differential forms $\Theta=(\theta_1,\ldots,\theta_N)$ on $V$ is chosen and assign to
every $\theta_{\alpha}\in\Omega^*(V)$ a formal variable $t^{\alpha}$
with grading
\begin{equation*} |t^{\alpha}|=2 -\deg\theta_{\alpha}. \end{equation*}
With this let $\PP$ be the Poisson algebra of formal power series in the variables $p_{\gamma}$ and $t^\alpha$ with
coefficients which are polynomials in the variables $q_{\gamma}$ and Laurent series in the Novikov ring variables $z_i$, with Poisson bracket given by
\begin{equation*}
 \{f,g\} = \sum_{\gamma}\kappa_{\gamma}\Bigl(\frac{\del f}{\del p_{\gamma}}\frac{\del g}{\del q_{\gamma}} -
                          (-1)^{|f||g|}\frac{\del g}{\del p_{\gamma}}\frac{\del f}{\del q_{\gamma}}\Bigr).
\end{equation*}
where $\kappa_\gamma$ is the multiplicity of the orbit $\gamma$.\\

Consider the union $\CM_{r,n^+,n^-,A}$ of all moduli spaces $\CM_{r,A}(\Gamma^+,\Gamma^-)$ for $|\Gamma^+|=n^+$ and $|\Gamma^-|=n^-$. As in Gromov-Witten theory we want to organize all moduli spaces $\CM_{r,n^+,n^-,A}$
into a generating function $\Ih\in\PP$, called {\it Hamiltonian}. In order to include also higher-dimensional
moduli spaces, in \cite{EGH} the authors follow the approach in Gromov-Witten theory to integrate the chosen differential forms
$\theta_{\alpha}$ and $[\gamma]$ (the canonical basis of $H^*(\mathcal{P})$) over the moduli spaces after pulling them back under the evaluation maps $\ev_i$, $i=1,\ldots,r$, $ev_{\pm,j}$, $j=1,\ldots,n^\pm$ at the marked points and punctures to the target manifold $V$ and the space of (positive or negative) good Reeb orbits $\mathcal{P}$, respectively. Consider furthermore evaluation maps $\ev_{\pm\infty,j}:\CM\to \bigcup_{\gamma\in\mathcal{P}}\bar\gamma$, $j=1,\ldots,n^\pm$ defined by the asymptotic markers at each puncture. Let
$$t=\sum_{\alpha=1}^N t^\alpha \theta_\alpha$$
$$p=\sum_{\gamma\in\mathcal{P}} \frac{1}{\kappa_\gamma}p_\gamma [\gamma]$$
$$q=\sum_{\gamma\in\mathcal{P}} \frac{1}{\kappa_\gamma}q_\gamma [\gamma]$$
The Hamiltonian $\Ih$ is then defined by
\begin{equation*}
\begin{split}
 \Ih = \sum_{\substack{r,A,\\n^+,n^-}} \frac{1}{r! n^+! n^-!} \int_{\CM_{r,n^+,n^-,A}/\IR}&
\ \bigwedge_{i=1}^r\ev_i^*t\ \bigwedge_{j=1}^{n^+}(\ev_{+,j}^* p \wedge \ev^*_{+\infty,j} d\phi_{\bar\gamma^+_j}) \wedge \\
& \bigwedge_{k=1}^{n^-}(\ev_{-,k}^*q\wedge \ev^*_{-\infty,k} d\phi_{\bar\gamma^-_k})\ z^A
\end{split}
\end{equation*}
with $z^A = z_0^{d_0} \cdot \ldots \cdot z_M^{d_M}$ for $A=d_0A_0+\ldots+d_M A_M$.
The index formula for the dimension of the moduli space $\CM_{r,n^+,n^-,A}$ \cite{EGH} implies that $ |\Ih|=2(m-3)-1$.
\vspace{0.5cm}

\subsection{Gravitational descendants}
We recall the definition of gravitational descendants in symplectic field theory (see \cite{F2}). In complete analogy to Gromov-Witten theory we can introduce $r$ tautological line bundles $\LL_1,\ldots,\LL_r$ over each moduli space $\CM_{r,A}(\Gamma^+,\Gamma^-)$, as the pull-back of the relative dualizing sheaf of $\pi_i: \CM_{r+1,A}(\Gamma^+,\Gamma^-)\to\CM_{r,A}(\Gamma^+,\Gamma^-)$ under the canonical section $\sigma_i: \CM_{r,A}(\Gamma^+,\Gamma^-)\to\CM_{r+1,A}(\Gamma^+,\Gamma^-)$ mapping to the $i$-th marked point in the fibre. \\

As in Gromov-Witten theory we would like to consider the integration of (powers of) the first Chern class
of the tautological line bundles over the moduli space, which by Poincar\'e duality corresponds to counting common zeroes of
sections of such bundles. However, in symplectic field theory the moduli spaces can have
codimension-one boundary, so we need to replace integration of the first Chern class of the tautological line bundle over a single moduli
space with a construction involving all moduli space at once, which preserves the algebraic structure of SFT. \\

Following the compactness statement in [BEHWZ] and assuming transversality and gluing, we have already said that the codimension-one boundary of a moduli space $\CM=\CM_{r,A}(\Gamma^+,\Gamma^-)$ of SFT holomorphic curves consists of curves with two levels. More precisely, each component of the boundary has the form of a fibred product $\CM_1\times\CM_2=\CM_{r_1,A_1}(\Gamma_1^+,\Gamma_1^-)$ $\phantom{.}_{\ev_{-,n_1^-}}\times_{\ev_{+,1}}\CM_{r_2,A_2}(\Gamma_2^+,\Gamma_2^-)$ of moduli spaces (of strictly lower dimension), quotiented by the $T^2$-action that rotates asymptotic markers at the connecting puncture, and with the marked points distributed on the two levels. Consider a boundary component where the $i$-th marked point sits, say, on the first level $\CM_1$: it directly follows from the definition of the tautological line bundle $\LL_i$ at the $i$-th marked point over $\CM$ that, over such boundary component,
\begin{equation*} \pi^*\LL_i|_{\CM_1\times\CM_2/T^2} = \pi_1^*\LL_{i,1}\end{equation*}
where $\LL_{i,1}$ denotes the tautological line bundle over the moduli space $\CM_1$, $\pi:\CM_1\times\CM_2\to\CM_1\times\CM_2/T^2$ is the canonical projection to the quotient and $\pi_1:\CM_1\times\CM_2\to\CM_1$ is the projection onto the first factor. With this we can now give the definition of coherent collections of sections in tautological line bundles from \cite{F2}.

\begin{definition} Assume that we have chosen sections $s_i$ in the tautological line bundles $\LL_i$ over all moduli spaces $\CM$ of $\Ju$-holomorphic curves of SFT. Then these collections of sections $(s_i)$ are called \emph{coherent} if for every section $s_i$ of $\LL_i$ over a moduli space $\CM$ the following holds: for each codimension-one boundary component $\CM_1\times\CM_2/T^2$ of $\CM$ we have $\pi^*(s_i|_{\CM_1\times\CM_2/T^2})=\pi_1^*s_{i,k}$ with $k=1,2$  for the corresponding section $s_{i,k}$ of the tautological line bundle $\LL_{i,k}$ over $\CM_k$, assuming that the $i$-th marked point sits on the $k$-th level.\end{definition}

Since in the end we will again be interested in the zero sets of these sections, we will assume that all occuring sections are sufficiently generic, in particular, transversal to the zero section. Furthermore, we want to assume that all the chosen sections are indeed invariant under the obvious symmetries like reordering of punctures and marked points. In order to meet both requirements, it follows that we actually need to employ multi-sections (in the sense of branched manifolds). On the other hand, it is clear that one can always find coherent collections of (transversal) sections by using induction on the dimension of the underlying moduli space: indeed, as was remarked already in \cite{F2}, existence of coherent collections of sections is guaranteed thanks to the fact that, in the smooth category, it is always possible to extend a section of a bundle from a lower dimensional submanifold of the base to the full space. The idea is then to start by choosing sections on the lowest dimensional moduli spaces and work our way up to the bigger dimensional ones by extending such sections from the boundaries to the interiors of $\CM$.\\

For every tuple $(j_1,\ldots,j_r)$ of natural numbers we choose $j_i$ coherent collections of sections $(s_{i,k})$ of $\LL_i$. Then we define for every moduli space $\CM=\CM_{r,A}(\Gamma^+,\Gamma^-)$,
\begin{equation*}
\CM^{(j_1,...,j_r)} = s_{1,1}^{-1}(0)\cap\ldots\cap s_{1,j_1}^{-1}(0)\cap\ldots\cap s_{r,1}^{-1}(0)\cap\ldots\cap s_{r,j_r}^{-1}(0)\subset \CM.
\end{equation*}
Note that by choosing all sections sufficiently generic, we can assume $\CM^{(j_1,...,j_r)}=\CM_{r,A}^{(j_1,...,j_r)}(\Gamma^+,\Gamma^-)$ is a branched-labelled submanifold of the moduli space $\CM_{r,A}(\Gamma^+,\Gamma^-)$. Note that by definition
\begin{equation*} \CM^{(j_1,\ldots,j_r)} = \CM^{(j_1,0,\ldots,0)}\cap \ldots \cap \CM^{(0,\ldots,0,j_r)}, \end{equation*}
and it follows from the coherency condition that the codimension-one boundary of $\CM^{(0,\ldots,0,j,0,\ldots,0)}$ is given by the products $\CM^{(0,\ldots,0,j,0,\ldots,0)}_1\times\CM_2/T^2$ or $\CM_1\times\CM^{(0,\ldots,0,j,0,\ldots,0)}_2/T^2$ (depending on whether the $i$-th marked point sits on the first or second level).\\

With this we can define the descendant Hamiltonian of SFT, which we will denote by $\tilde{\Ih}$, while the
Hamiltonian $\Ih=\tilde{\Ih}|_{t^{\alpha,j}=0,j>0}$ defined in \cite{EGH} will from now on be called {\it primary}. In order to keep track of the descendants we
will assign to every chosen differential form $\theta_\alpha$ now a sequence of formal variables $t^{\alpha,j}$ with grading
\begin{equation*} |t^{\alpha,j}|=2(1-j) -\deg\theta_\alpha \end{equation*} and form an extended Poisson algebra $\tilde{\PP}$ accordingly.
Then the descendant Hamiltonian $\tilde{\Ih}\in\tilde{\PP}$ of (rational) SFT is defined by
\begin{equation*}
\begin{split}
 \tilde{\Ih} = \sum_{\substack{r,A,I\\n^+,n^-}} \int_{\CM^{(j_1,\ldots,j_r)}_{r,n^+,n^-,A}/\IR} &
 \ev_1^*\theta_{\alpha_1}\wedge\ldots\wedge\ev_r^*\theta_{\alpha_r} \bigwedge_{j=1}^{n^+}(\ev_{+,j}^* p \wedge \ev^*_{+\infty,j} d\phi_{\bar\gamma^+_j}) \wedge \\
& \bigwedge_{k=1}^{n^-}(\ev_{-,k}^*q\wedge \ev^*_{-\infty,k} d\phi_{\bar\gamma^-_k})\; t^I z^A,
\end{split}
\end{equation*}
where $t^I=t^{\alpha_1,j_1} \ldots t^{\alpha_r,j_r}$ and $z^A = z_0^{d_0} \cdot \ldots \cdot z_M^{d_M}$ for $A=d_0A_0+\ldots+d_M A_M$.\\

\vspace{0.5cm}

\subsection{Hamiltonian systems with symmetries}
Symplectic field theory assigns to every contact manifold not only a Poisson algebra, rational SFT homology, but also, thanks to gravitational descendants, a Hamiltonian system in it with an infinite number of symmetries.

\begin{theorem} Differentiating the rational Hamiltonian $\tilde{\Ih}\in\tilde{\PP}$ with respect to the formal variables $t_{\alpha,p}$
defines a sequence of classical Hamiltonians
\begin{equation*} \tilde{\Ih}_{\alpha,p}=\frac{\del\tilde{\Ih}}{\del t^{\alpha,p}} \in H_*(\tilde{\PP},\{\tilde{\Ih},\cdot\}) \end{equation*}
{\it in the rational SFT homology algebra with differential $\tilde{d}=\{\tilde{\Ih},\cdot\}: \tilde{\PP}\to\tilde{\PP}$, which commute with respect to the
bracket on $H_*(\tilde{\PP},\{\tilde{\Ih},\cdot\})$,}
\begin{equation*} \{\tilde{\Ih}_{\alpha,p},\tilde{\Ih}_{\beta,q}\} = 0,\; (\alpha,p),(\beta,q)\in\{1,\ldots,N\}\times\IN. \end{equation*}
\end{theorem}

Everything is an immediate consequence of the master equation $\{\tilde{\Ih},\tilde{\Ih}\}=0$, which can be proven in the same
way as in the case without descendants using the results in \cite{F2}. The idea is simply using Stokes formula on the boundary of each moduli space appearing in the definition of the Hamiltonian. Such integral over ($\kappa_\gamma$ times) a boundary components of the form $\CM_1\times\CM_2/T^2$ can be expressed as an integral over the product $\CM_1\times\CM_2$ by integrating an extra $2$-form which is the pull back of $\frac{1}{k_\gamma}(d\phi_{\bar\gamma^+}\wedge d\phi_{\bar\gamma^-})$ at the connecting Reeb orbit $\gamma$ (notice how the $\kappa_\gamma^2$, that is produced by integration of such differential form over the fibres of the $T^2$-quotient projection, compensates with the $1/\kappa_\gamma$ factor to give the correct contribution $\kappa_\gamma$). On the other hand, this integral must give zero by Stokes theorem, as we are integrating closed differential forms on boundary $\del \CM$. This results is expressed, using the above Poisson algebra formalism, as $\{\tilde{\Ih},\tilde{\Ih}\}=0$.\\

Now the boundary equation $\tilde{d}\circ \tilde{d}=0$, $\tilde{d}=\{\tilde{\Ih},\cdot\}$ follows directly from the identity $\{\tilde{\Ih},\tilde{\Ih}\}=0$, while the fact that every $\tilde{\Ih}_{\alpha,p}$, $(\alpha,p)\in
\{1,\ldots,N\}\times\IN$ defines an element in the homology $H_*(\tilde{\PP},\{\tilde{\Ih},\cdot\})$ follows from the identity
\begin{equation*} \{\tilde{\Ih},\tilde{\Ih}_{\alpha,p}\} = 0,\end{equation*}
which can be shown by differentiating the master equation with respect to the $t^{\alpha,p}$-variable and using the
graded Leibniz rule,
\[ \frac{\del}{\del t^{\alpha,p}} \{f,g\} =
   \{\frac{\del f}{\del t^{\alpha,p}},g\} + (-1)^{|t^{\alpha,p}||f|} \{f,\frac{\del g}{\del t^{\alpha,p}}\}. \]
On the other hand,the fact that any two $\tilde{\Ih}_{\alpha,p}$,
$\tilde{\Ih}_{\beta,q}$ commute {\it after passing to homology} follows from the identity
\begin{equation*}
 \{\tilde{\Ih}_{\alpha,p},\tilde{\Ih}_{\beta,q}\}+(-1)^{|t^{\alpha,p}|}\{\tilde{\Ih},\frac{\del^2\tilde{\Ih}}{\del t^{\alpha,p}\del t^{\beta,q}}\} = 0.
\end{equation*}
obtained by differentiating the master equation twice and by recalling that $\tilde{\Ih}$ is homogeneous of odd degree.\\

We now turn to the question of independence of these nice algebraic structures from the choices like contact form,
cylindrical almost complex structure, representatives for the classes $[\theta_\alpha]\in H^*(V)$ and $[d\phi_{\bar\gamma}]\in H^*(S^1)$, abstract polyfold perturbations and, of course, the choice of the coherent
collection of sections. This is the content of the following theorem proven in \cite{F2}.

\begin{theorem} For different choices of contact form $\lambda^{\pm}$, cylindrical almost complex structure
$\Ju^{\pm}$ , representatives for the classes $[\theta_\alpha]\in H^*(V)$ and $[d\phi_{\bar\gamma}]\in H^*(S^1)$, abstract polyfold perturbations and sequences of coherent collections of sections $(s^{\pm}_j)$ the
resulting systems of commuting functions $\tilde{\Ih}^-_{\alpha,p}$ on $H_*(\tilde{\PP}^-,d^-)$ and $\tilde{\Ih}^+_{\alpha,p}$ on
$H_*(\tilde{\PP}^+,\tilde{d}^+)$ are isomorphic, i.e. there exists an isomorphism of the Poisson algebras $H_*(\tilde{\PP}^-,\tilde{d}^-)$
and $H_*(\tilde{\PP}^+,\tilde{d}^+)$ which maps $\tilde{\Ih}^-_{\alpha,p}\in H_*(\tilde{\PP}^-,\tilde{d}^-)$ to
$\tilde{\Ih}^+_{\alpha,p}\in H_*(\tilde{\PP}^+,\tilde{d}^+)$.
\end{theorem}

This theorem is an immediate extension of the theorem in \cite{EGH} which states that for different choices of auxiliary data the Poisson algebras $H_*(\PP^-,d^-)$ and $H_*(\PP^+,d^+)$ with $d^{\pm}=\{\Ih^{\pm},\cdot\}$ are isomorphic. In particular the extension in \cite{F2} is about the coherent collection of sections $(s^{\pm}_j)$.  Here one needs a notion of a collection of sections $(s_j)$ in the tautological line bundles over all moduli spaces of holomorphic curves in the cylindrical cobordism interpolating between the auxiliary structures, which are {\it coherently connecting} the two coherent collections of sections $(s^{\pm}_j)$.\\

We want to point out the fact that the primary Poisson SFT homology algebra can be thought of as the space of functions on some abstract infinite-dimensional Poisson super-space. Indeed, consider the Poisson super-space $\mathbf{V}$ underlying the Poisson algebra $\PP$. Then the kernel $\ker(\{\Ih,\cdot\})$ can be seen as the algebra of functions on the space $\mathcal{O}$ of orbits in $\mathbf{V}$ of the Hamiltonian $\IR$-action given by $\Ih$, that is, the flow lines of the Hamiltonian vector field $X_{\Ih}$ associated to $\Ih$. Even in a finite dimensional setting the space $\mathcal{O}$ can be very wild. Anyhow the image $\text{im}(\{\Ih,\cdot\})$ is an ideal of such algebra and hence identifies a sub-space of $\mathcal{O}$ given by all of those orbits $o\in \mathcal{O}$ at which, for any $f\in\PP$, $\{\Ih,f\}|_{o}=0$. But such orbits are simply the constant ones, where $X_{\Ih}$ vanishes. Hence the Poisson SFT-homology algebra $H_*(\PP,\{\Ih,\cdot\})$ can be regarded as the algebra of functions on $X_{\Ih}^{-1}(0)$, seen as a subspace of the space $\mathcal{O}$ of orbits of $\Ih$, endowed with a Poisson structure by singular, stationary reduction. In particular the descendant Hamiltonians $\Ih_{\alpha,j}:=\left.\frac{\del \tilde{\Ih}}{\del t^{\alpha,j}}\right|_{t^{\alpha,j}=0,j>0} \in H_*(\PP,\{\Ih,\cdot\})$ are examples of functions on such space.\\

Finally we recall a result from \cite{FR} that states that, besides commutativity, the SFT Hamiltonians satisfy analogues of the well-known string, dilaton and divisor equations of Gromov-Witten theory. Such equations hold, after passing to SFT homology, for any auxiliary choice used to define the Hamiltonians.

\begin{theorem}\label{SDD}
 For any choice of differential forms and coherent sections the following \emph{string, dilaton and divisor equations} hold \emph{after} passing to SFT-homology
\begin{eqnarray*}
\frac{\del}{\del t^{1,0}}\tilde{\Ih} &=& \int_V t\wedge t  + \sum_{k}t^{\alpha,k+1}\frac{\del}{\del t^{\alpha,k}}\tilde{\Ih}
 \;\in\; H_*(\tilde{\PP},\{\tilde{\Ih},\cdot\}), \\
\frac{\del}{\del t^{1,1}}\tilde{\Ih} &=& \ID_{\mathrm{Euler}}\tilde{\Ih}  \;\in\, H_*(\tilde{\PP},\{\tilde{\Ih},\cdot\}), \\
\left(\frac{\del}{\del t^{2,0}} -z\frac{\del}{\del z}\right)\tilde{\Ih} &=& \int_V t\wedge t\wedge \theta_2 + \sum_{k} t^{\alpha,k+1} c_{2\alpha}^\beta\frac{\del\tilde{\Ih}}{\del t^{\beta, k}} \;\in\; H_*(\tilde{\PP},\{\tilde{\Ih},\cdot\}),
\end{eqnarray*}
where $t^{1,k}$ is the $t$-variable associated to the $k$-th descendant of the unity class $1\in H^*(V)$, $t^{2,k}$ is the one associated with $\theta_2 \in H^2(V)$ and $z$ the corresponding Novikov ring variable, and $\ID_{\mathrm{Euler}}$ is the linear differential operator
$$\ID_{\mathrm{Euler}} := 2 -\sum_\gamma p_\gamma\frac{\del}{\del p_\gamma}
-\sum_\gamma q_\gamma\frac{\del}{\del q_\gamma}-\sum_{\alpha,p}t^{\alpha,p}\frac{\del}{\del t^{\alpha,p}}.$$\\
\end{theorem}

\begin{remark}\label{noncontact}
Most of the above results actually hold in the case of more general stable Hamiltonian structures, beyond the contact situation, but their precise statement and the involved construction require some variations. One of the main problems arises from the fact that, in the non-contact case, the coefficients of monomials in the variables $p_\gamma$ in the Hamiltonian $\Ih$ might not be polynomials in the $q_\gamma$ variables. The  Hamiltonian $\Ih$ would then fail to be an element of the Poisson algebra $\PP$ as we have defined it and the construction of invariants would be more subtle. In the following we will mostly stick to the contact case, but whenever we consider target manifolds that have a non-contact stable Hamiltonian structure, we will assume that these problems do not arise and the above contructions remain valid.
\end{remark}

\begin{remark}\label{rigour}
We end this introductory section with a short discussion about the level of rigour of our results. All the algebraic results we present for holomorphic curves rely first of all on the fact that all appearing moduli spaces are (weighted branched) manifolds (or orbifolds) with corners, of dimension equal to the Fredholm index of the Cauchy-Riemann operator. In order to equip the zero set of the Cauchy-Riemann operator with such structure, one would like to apply an infinite-dimensional version of the classical implicit function theorem, and the crucial step here is to prove a transversality result for the Cauchy-Riemann operator. While it is well-known that transversality holds for a generic choice of almost complex structure as long as all holomorphic curves are simple, several problems appear when the curve is multiply-covered. This problem is already present in (symplectic) Gromov-Witten theory and Floer homology, and very involved tools like virtual moduli cycles, Kuranishi structures and polyfolds were developed to solve it. In particular the polyfold approach of Hofer, Wysocki and Zehnder, see \cite{HWZ} is supposed to solve all the challenges in the most satisfactory way (see also the survey \cite{F3}), but it is not yet fully completed. Since they promise to prove transversality for symplectic field theory and Gromov-Witten theory in one of their upcoming papers, we follow other papers in the field in considering everything up to transversality and state it nevertheless as a theorem. However, beside not dealing with these foundational problems, the proofs we give in this paper are admittedly somewhat sketchy. However we claim that most missing technical details (apart transversality), with special emphasis on the non-$S^1$-equivariant moduli spaces we consider throughout the paper, can be recovered in the literature. In particular this paper uses the description of  moduli spaces of $S^1$-parametrized cylinders with punctures which appered in \cite{BO}, where a version of non-$S^1$-equivariant contact homology is presented, to shed some light on the elegant algebraic structure that can be deduced from it. Even when we introduce some other versions of moduli spaces of parametrized curves (in the proof of theorems \ref{masterCH} and \ref{omegarecursionCH}, for instance) and describe what form their compactification (and in particular codimension-$1$ boundary) should have, we use the same type of curve degeneration (the matching of the $S^1$-parametrization at the corresponding node/Reeb orbit). In the final part of the paper, when we move to the case of rational SFT, we use again the same ideas to sketch how some analogous definition of non-$S^1$-equivariant correlators produce an algebraic structure so natural that it integrates perfectly with the Hamiltonian systems formalism of the theory (even pointing to an answer to the integrability problem for such systems, for which for now no approach was found). We would like, in any case, to stress the ``research announcement'' character, especially for the second part of the paper, as we leave a more attentive analysis of the geometry of the moduli spaces (again, even apart from transversality) for a future work.
\end{remark}

\section{$N$-recursion for contact homology}

\subsection{Contact homology}
Contact homology is a reduced version of symplectic field theory that only uses curves with at most one positive puncture. In case the target manifold $V$ is contact, the maximum principle for holomorphic curves in symplectizations forbids that a SFT-holomorphic curve with $s^+$ positive punctures can degenerate to a multi-floor curve where any of the components has more than $s^+$ positive punctures. The absence of local maxima further implies that any curve with no positive punctures must be constant. In particular this means that we can study moduli spaces of SFT-curves with only one positive puncture and be safe that the boundary only involves moduli spaces of the same type. Besides, in such moduli spaces nodal degeneration only involve the appearance of constant bubbles (when different marked points come together).\\

The algebraic structure we obtain is the linear part in the $p$-variables of the one described for full SFT. In particular we get a complex formed by the graded commutative algebra $\mathcal{A}$ generated by the variables $q_\gamma$ over the power series in the variables $t^{\alpha}=t^{\alpha,0}$ with coefficients in the Novikov ring of variables $z_k$ (i.e. the evaluation at $p=0$ of the Poisson algebra $\PP$), and differential given by the (odd) vector field $$X=\sum_\gamma X^\gamma\frac{\del}{\del q_\gamma}=\sum_\gamma \kappa_\gamma \left.\frac{\del \Ih}{\del p_{\gamma}}\right|_{p=0} \frac{\del}{\del q_\gamma}.$$
The master equation $\{\Ih,\Ih\}=0$ reduces to $[X,X]=0$ (the square bracket here stands for graded Lie bracket on graded vector fields and, from the formula above and the fact that $|\Ih|=2(m-3)-1$ is odd, we obtain that $|X|=-1$, so $X$ is an odd vector field), the resulting homology will be denoted by $CH(\mathcal{A};X)$ (or simply $CH(V)$ when there is no danger of confusion) and the system of commuting (on SFT-homology) descendant Hamiltonians $\Ih_{\alpha,i}=\left.\frac{\del \tilde{\Ih}}{\del t^{\alpha,i}}\right|_{t^{\beta,j}=0,j>0}$, $(\alpha,i)\in\{1,\ldots,N\}\times\IN$ induce a system of Lie-commuting vector fields on contact homology $$X_{\alpha,i}=\sum_{\gamma}\kappa_\gamma\left.\frac{\del \Ih_{\alpha,i}}{\del p_\gamma}\right|_{p=0}\frac{\del}{\del q_\gamma} :CH(V)\to CH(V).$$\\

\subsection{The non-equivariant differential revisited}
Consider now a moduli spaces of punctured $S^1$-parametrized cylinders with marked points. We start with the fully parametrized space  $\IM^{S^1,0}_{r,A}(\gamma_0,\gamma_\infty,\Gamma^-)$ consisting of tuples $(u,(z_k^{-}),(z_i))$, where $(z^{-}_k),(z_i)$ are two disjoint ordered sets of points on $\IP^1 \setminus (\{0,\infty\})=S^1\times\IR$ (namely negative punctures with asymptotic markers, and $r$ additional marked points). The map $u: \Si \to \IR \times V$  from the punctured Riemann surface $\Si = \IP^1 \setminus (\{0,\infty\} \cup  \{(z_k^{-})\})$ is required to satisfy the Cauchy-Riemann equation
\begin{equation*}
 \CR_{\Ju} u = du + \Ju(u) \cdot du \cdot i = 0
\end{equation*}
with the complex structure $i$ on $\IP^1$. Assuming we have chosen cylindrical coordinates $\psi^{-}_k: \IR^{-}
\times S^1 \to \Si$ around each puncture $z^{-}_k$, in the sense that $\psi_k^{-}(-\infty,t)=z_k^{-}$,
the map $u$ is additionally required to show for all $k=1,...,n^{-}$ the
asymptotic behaviour
\begin{equation*}
 \lim_{s\to -\infty} (u \circ \psi^{-}_k) (s,t+t_0) =
 (- \infty,\gamma^{-}_k(T^{-}_kt))
\end{equation*}
with $t_0\in S^1$ and with orbits $\gamma^{-}_k\in \Gamma^{-}$, where $T^{-}_k>0$ denotes period of $\gamma^{-}_k$, and analogous asymptotic behaviour at $0$ and $\infty$ for $s\to +\infty$ and $s\to -\infty$ respectively, for orbits $\gamma_0$ and $\gamma_\infty$ and with respect to the natural coordinates on $S^1\times\IR$. We assign to each curve an absolute homology class $A$ employing as usual a choice of spanning surfaces.
In order to obtain the $S^1$-parametrized space $\IM^{S^1}_{r,A}(\gamma_0,\gamma_\infty,\Gamma^-)$ we only divide out the $\IR$-component of the $S^1\times \IR$ group of automorphisms of the cylinder $\IP^1 \setminus (\{0,\infty\})$ and, as usual, the $\IR$-action coming from the cylindrical target $V\times\IR$ as well.\\

This moduli space can be compactified to $\CM^{S^1}=\CM^{S^1}_{r,A}(\gamma_0,\gamma_\infty,\Gamma^-)$ by adding moduli spaces of $S^1$-parametrized multi-floor curves with ghost bubbles, where the puncture at $0$ is always the positive puncture of the top floor, the puncture at $\infty$ can be on any floor and the $S^1$-parametrization is remembered when going through connecting punctures as explained in \cite{BEHWZ} (compactification of the space of curves with {\it decorations}). More explicitly, in such compactification, a $n$-floor curve with the $\infty$-puncture on its $k$-th component from the top, has the upper $k$ components that are $S^1$-parametrized curves where, at each connecting puncture, the $S^1$-coordinates of different components match, while the lower $(n-k)$ are non-$S^1$-parametrized. Also, both type of components possibly have stable constant bubbles.\\

The space $\CM^{S^1}$ carries, besides the usual evaluation maps at marked points, orbits and asymptotic markers, also evaluation maps at the special punctures at $0$ and $\infty$ to the corresponding target simple Reeb orbits given by the special $S^1$-coordinate on the curve,
$$\ev_{+\infty,0}:\CM^{S^1}\to\bar\gamma_0 \simeq S^1$$
$$\ev_{-\infty,\infty}:\CM^{S^1}\to  \bar\gamma_\infty \simeq S^1.$$\\
These two evaluation maps are similar to the ones defined at any other puncture by asymptotic markers, but are now coupled thank to the fixed parametrization.\\

We form the space $\CM^{S^1}_{r,n^-,A}(\gamma_0,\gamma_\infty)$ by taking the union over $\Gamma^-$ of all the spaces $\CM^{S^1}_{r,A}(\gamma_0,\gamma_\infty,\Gamma^-) $ with $|\Gamma^-|=n^-$.\\

This moduli space was actually already introduced in \cite{BO} to define the non-$S^1$-equivariant linearized contact homology differential. We will proceed in a similar way defining a $(1,1)$-tensor  $N(t^\alpha)$ depending on parameters $t^1,\ldots,t^N$, on the super-space $\mathbf{Q}$ underlying the algebra $\mathcal{A}$. A point $q\in\mathbf{Q}$ is a cohomology class $q=\sum_{\gamma\in\mathcal{P}} \frac{1}{\kappa_\gamma}q_\gamma [\gamma]$ on $\mathcal{P}$, with the notations of section $1$. We will write $q^\gamma$ instead of $q_\gamma$ to be coherent with the notion that $q_\gamma$ is treated here as a coordinate for the space $\mathbf{Q}$, while we will treat the $t$-variables as parameters on which the functions on $\mathbf{Q}$ can depend. Using such coordinates we define
$$N=N_{\gamma_1}^{\gamma_2}(t,q)\  dq^{\gamma_1}\otimes \frac{\del}{\del q^{\gamma_2}}$$
where we sum over repeated indices, and
\begin{equation*}
\begin{split}
N_{\gamma_1}^{\gamma_2}(t,q) =\sum  \frac{1}{r! n^- ! \kappa_{\gamma_1}} \int_{\CM^{S^1}_{r,n^-,A}(\gamma_1,\gamma_2)} & \bigwedge_{i=1}^r \ev_i^* t\ \bigwedge_{j=1}^{n^-}(ev^*_{-,j}q \wedge \ev^*_{-\infty,j}d\phi_{\bar\gamma^-_j})\ \wedge \\
&\wedge \ev_{+\infty,0}^* d\phi_{\bar\gamma_1} \wedge \ev_{-\infty,\infty}^* d\phi_{\bar\gamma_2}
\end{split}
\end{equation*}
With the usual grading of the SFT variables, and assigning degree $0$ to the exterior differential $d$ on the superspace $\mathbf{V}$, from the index formula for the dimension of the moduli space of SFT-curves, we deduce that N has even degree:
$$|N|=-2$$
This comes from the fact our moduli space has dimension one less than the ordinary, $S^1$-equivariant, moduli space of curves involved in the definition of the contact homology vector field $X$, because of the extra constraint we are imposing when we want the asymptotic markers at $0$ and $\infty$ to be coupled (so the degree of $N$ is one less than the degree of $X$).\\

Let us now briefly describe the codimension one boundary of the moduli space $\CM^{S^1}$, assuming transversality and following \cite{BO}. Such boundary consists of $2$-floor curves of one of the following two types. Either the the connecting node/Reeb orbit $\gamma$ separates the $0$- and $\infty$-puncture and the global $S^1$-parametrization is remembered when going through such node (meaning that the parametrizations of the two floors match when meeting at $\gamma$) or the $0$- and $\infty$-puncture are both on the top floor, which is then again an $S^1$-parametrized cylinder (with punctures and marked points), while the bottom floor is an ordinary unparametrized curve of the type usually studied in ordinary, $S^1$-equivariant, contact homology. More precisely, denoting these two components of the boundary $\del'_{(r_1,A_1,\Gamma_1|\gamma|r_2,A_2,\Gamma_2)} \CM^{S^1}$ and $\del''_{(r_1,A_1,\Gamma_1|\gamma|r_2,A_2,\Gamma_2)} \CM^{S^1}$ respectively (the subscript indicates how marked points, homology class and negative punctures distribute among the two floors, $\Gamma^-=\Gamma^-_1\cup \Gamma^-_2$, $r=r_1+r_2$, $A=A_1+A_2$), we have the identifications:
$$\del' \CM^{S^1} \simeq \CM^{S^1}_{r_1,A_1}(\gamma_0,\gamma,\Gamma^-_1)\  \prescript{}{\ev_{-\infty,\infty}}\times_{\ev_{+\infty,0}} \ \CM^{S^1}_{r_2,A_2}(\gamma,\gamma_\infty,\Gamma^-_2)$$
$$\del'' \CM^{S^1}\simeq(\CM^{S^1}_{r_1,A_1}(\gamma_0,\gamma_\infty,\Gamma^-_1\cup\{\gamma\})\times \CM_{r_2,A_2}(\gamma,\Gamma^-_2)) /T^2$$

Notice that the no-descendant (or primary) contact homology differential $X$ (still parametrized by the primary variables $t^\alpha$) induces a differential $\mathcal{L}_{X}$ (Lie derivative along the vector field $X$) on the space of $(k,l)$-tensor fields $\mathcal{T}^{(k,l)}\mathbf{Q}$ (again with parameters $t^\alpha$) on the super-space $\mathbf{Q}$. The resulting homology, which we denote $CH(\mathcal{T}^{(k,l)}\mathbf{Q};\mathcal{L}_X)$, is a module over $CH(\mathcal{A};X)=CH(\mathcal{T}^{(0,0)}\mathbf{Q};\mathcal{L}_{X})$ and is an invariant of the contact structure on $V$, as it can easily be proved with the same procedure as for $CH(\mathcal{A};X)$. In particular, for two different choices of contact form $\lambda^{\pm}$, cylindrical almost complex structure
$\Ju^{\pm}$ , representatives for the classes $[\theta_\alpha]\in H^*(V)$ and $[d\phi_{\bar\gamma}]\in H^*(S^1)$, abstract polyfold perturbations and sequences of coherent collections of sections $(s^{\pm}_j)$, there exist an isomorphism
$$d\varphi^\pm: CH(\mathcal{T}^{(k,l)}\mathbf{Q^+},\mathcal{L}_{X^+}) \to CH(\mathcal{T}^{(k,l)}\mathbf{Q^-},\mathcal{L}_{X^-})$$
which is simply the differential of the isomorphism 
$$\varphi^\pm: CH(\mathcal{A^+};X^+) \to CH(\mathcal{A^-};X^-),$$
constructed in \cite{EGH} by studying curves in the cobordims $W=\overrightarrow{V^+V^-}$ interpolating between the two different choices. This differential $d\varphi^\pm$ is in the sense of differential geometric lift to the tensor algebra of a diffeomorphism of our base formal manifolds $\mathbf{Q}^+$ and $\mathbf{Q}^-$: since the diffeomorphism $\phi^\pm$ passes to homology with respect to a derivation $X$, its differential passes to homology with respect to the Lie derivative $\mathcal{L}_X$ (see also the discussion on invariance for satellites in SFT from \cite{EGH}, which is completely analogous).

\begin{theorem}
$$\mathcal{L}_X N = 0$$
and, denoting by $N^\pm$ the two $(1,1)$-tensors resulting from two different choices of contact form $\lambda^{\pm}$, cylindrical almost complex structure
$\Ju^{\pm}$ , representatives for the classes $[\theta_\alpha]\in H^*(V)$ and $[d\phi_{\bar\gamma}]\in H^*(S^1)$, abstract polyfold perturbations and sequences of coherent collections of sections $(s^{\pm}_j)$,
\begin{equation*}
\begin{split}
d\varphi^\pm:  \ &CH(\mathcal{T}^{(1,1)}\mathbf{Q^+},\mathcal{L}_{X^+})\ \to\ CH(\mathcal{T}^{(1,1)}\mathbf{Q^-},\mathcal{L}_{X^-})\\
&\hspace{1.5 cm}N^+\hspace{1.1cm} \mapsto \hspace{1.1cm}N^-
\end{split}
\end{equation*}
so that $N\in CH(\mathcal{T}^{(1,1)}\mathbf{Q},\mathcal{L}_{X})$ is an invariant of the contact structure on $V$.
\end{theorem}
\begin{proof}
For the proof of the equation $\mathcal{L}_X N  = 0$ we need to apply Stokes theorem to codimension-one boundary of the moduli spaces involved in the definition of $N$, i.e. moduli spaces of punctured $\IP^1$ with a marked $\IR^+$ line connecting two punctures $0$ and $\infty$ mapped to orbits $\gamma_1$ and $\gamma_2$, at which we pull back $1$-forms from the underlying simple orbits. In the picture below we represent a moduli space by drawing the corresponding generic element (the curve with the marked red $\IR^+$ line) and we represent the constraining of the endpoints via $1$-forms by the small triangles (a triangle pointing towards an orbit $\gamma$ means that the red line is $S^1$-constrained at that orbit by integrating the pull-back of the form $d\phi_{\bar\gamma}$). Let us consider the three terms in the right hand side of the pictorial equation below. The first two terms correspond to curve degenerations forming the boundary components of type $\del'\CM^{S^1}$ described above. Here we use the fact that, when a curve splits at a puncture $\gamma$  through which the $\IR^+$ line passes, the $S^1$-parametrizations match and this constraint is expressed by pulling-back a representative of the diagonal class in $H^*(\gamma\times\gamma)$. Indeed, integrating a differential form over the fibered product $\del' \CM^{S^1}$ is equivalent to integrating over the cartesian product the same form times the pullback via evaluation maps of the diagonal form $\frac{1}{\kappa_\gamma}(d\phi_{\bar\gamma}\otimes 1 + 1\otimes d\phi_{\bar\gamma}$), which gives the constraints at the connecting orbit for the first two terms in the right-hand side of the picture. Here the factor $1/\kappa_\gamma$ compensates for the fact that we want the $S^1$-parametrizations to match in the source rational curves, whose punctures branch with order$\kappa_\gamma$ over the connecting orbit $\bar{\gamma}$.  The effect of pulling back the diagonal class at a connecting orbit is, hence, the appearence of two terms in the right hand side of the pictorial equation (corresponding to the two summands in the $1$-form), whose only difference is the direction of the triangle at the connecting orbit (integrating $(d\phi_{\bar\gamma}\otimes 1$ corresponds to constraining the upper $S^1$-parametrization or red line, while $1\otimes d\phi_{\bar\gamma}$  the lower one). Finally the third term represents boundary components of the type $\del'' \CM^{S^1}$, where the connecting orbit is disjoint from the red line representing the fixed $S^1$-parametrization.\\
\begin{center}
\includegraphics[height=9cm]{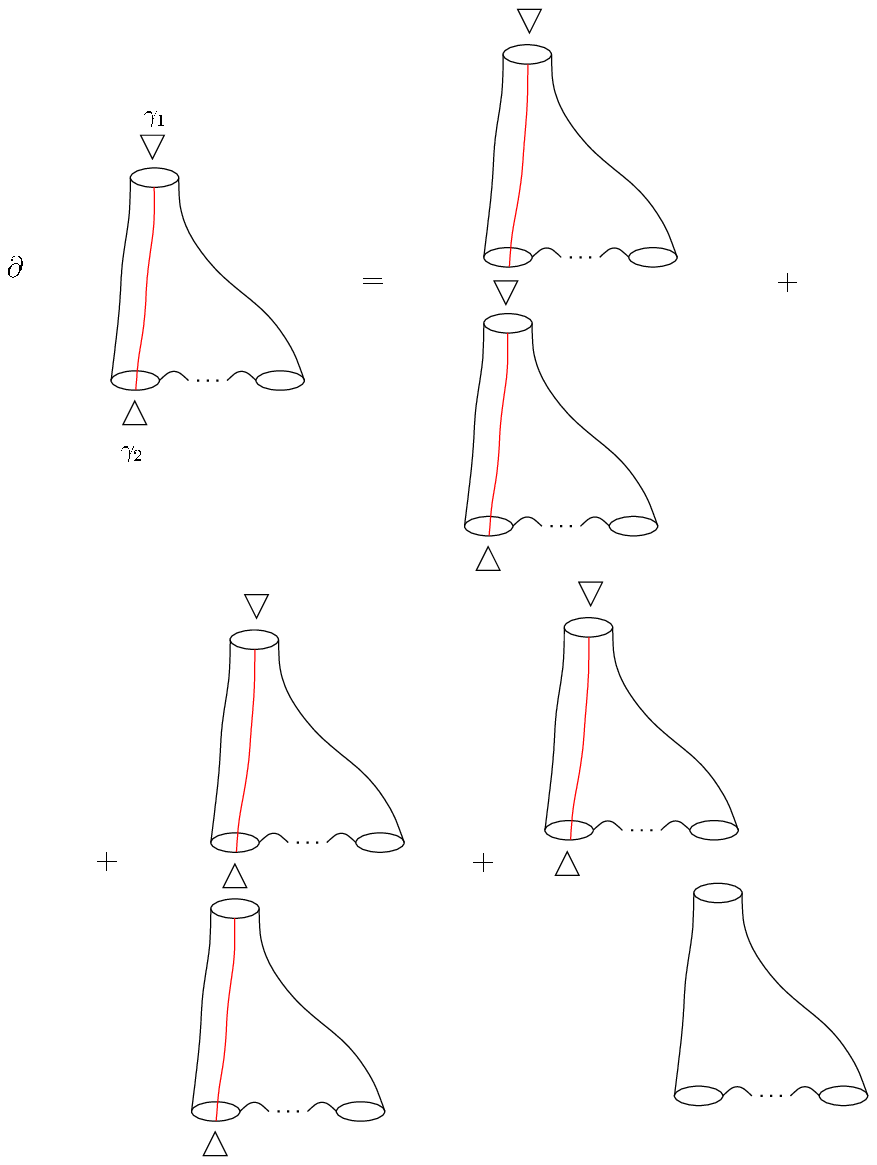}
\end{center}
All we need to notice, at this point, is that, taking orientation into account for the right signs, the three terms on the righ-hand side of the equation represented in the picture exactly correspond to the three summands in the coordinate expression of the (graded) Lie derivative $\mathcal{L}_XN$, which hence vanishes by Stokes theorem applied to the boundary of our moduli space.\\

For the second part of the theorem, about invariance with respect to auxiliary choices in the definition of $\mathcal{A}$, $X$ and $N$ a similar approach is needed, where we study the boundary of the moduli spaces of the same type of curves, but this time in the cobordism interpolating between two different choices of auxiliary data. Drawing again the same kind of pictures (only remembering, as explained in \cite{BEHWZ}, that the boundary of moduli spaces of connected curves in the cobordism is formed by $2$-floor curves in which one of the floors is a connected curve in the cylindrical manifold over one of the boundaries and the other is a possibly disconnected curve in the cobordism). Algebraically this gives precisely the transformation rule for $N$ described in the statement with respect to the lift $d\varphi^\pm$ of the isomorphism $\varphi^\pm$ to the homology tensor algebras of $\mathbf{Q^\pm}$.
\end{proof}

\begin{corollary}
For any $\alpha=1,\ldots,N$ and any $i=0,1,2,\ldots$
$$\mathcal{L}_{X_{\alpha,i}} N = 0\  \in CH(\mathcal{T}^{(1,1)}\mathbf{Q},\mathcal{L}_{X})$$
\end{corollary}
\begin{proof}
Simply expand the equation $\mathcal{L}_X N = 0$ in powers of the $t$-variables and consider the linear terms.
\end{proof}
\begin{example}
Consider the case $V=S^1$ with $t=t^1\theta_1 + \tau^1 \Theta_1$, $\theta_1=1$ and $\Theta_1=d\varphi$ where $\varphi$ is the angular coordinate on $S^1$.  It is easy to compute $\bar{N}:=N|_{\tau^1=0}$. Writing just $k$ for the index $k\gamma$ associated to the $k$-th multiple of the orbit $\gamma=V$, from the dimension formula for the moduli space of SFT-curves and an easy curve counting we immediately see that
$$\bar N_k^l=\frac{l-k}{k}q^{l-k},\hspace{1cm} l>k$$
$$\bar N_k^l=0,\hspace{2.2cm} l\leq k$$
Indeed, from dimension counting for $\CM^{S^1}$ we see that the only $0$-dimensional moduli spaces are those which contain branched covers of $\mathbb{P}^1$ with a single $l$-fold branch point over $\infty$ and two branch points, of branch numbers $k$ and $(l-k)$ over $0$ (we are of course identifying $S^1\times \IR$ with $\mathbb{P}^1\setminus \{0,\infty\}$). Once we have fixed these two zeros and one pole we have an $S^1$-worth of meromorphic functions to $\mathbb{P}^1$ (modulo real multiplicative factors, which is the $\IR$-action we are quotienting out), to which the $S^1$-parametrization (with respect to the positive and one of the negative punctures) and the position of the asymptotic marker over the other negative punctures are to be added. $\CM$ is hence a $S^1$-bundle over $T^2$. One $S^1$-degree of freedom is taken care of  by integrating $\ev_{-,1}^* d\varphi $ along the $S^1$-fibers (it gives a factor $(l-k)$, which cancels out with the denominator appearing in the definition $q=\sum_{\gamma\in\mathcal{P}} \frac{1}{\kappa_\gamma}q_\gamma [\gamma])$), while the further combinatorial factor $(l-k)$ comes from integrating the form $\ev_{+\infty,0}^* d\varphi \wedge \ev_{-\infty,\infty}^* d\varphi$ over the residual $T^2$ (this can be seen by reporting a fixed point $p\in V\times\{-\infty\}$ along a geodesic on the source $\mathbb{P}^1$ of our curve all the way to $V\times\{+\infty\}$: as the phase factor of our meromorphic function makes $k$ complete tours, the image of $p$ at  $V\times\{+\infty\}$ makes $(l-k)$ complete tours, so there are precisely $(l-k)$ phases for which a meridian from $0$ to $\infty$ in our source $\mathbb{P}^1$ is asymptotic to the same point $p$ at the two Reeb orbits at $\pm \infty$) and the denominator $k$ was directly in the definition of $N$.
\end{example}

\subsection{Vanishing of Nijenhuis torsion of $N$}
In a graded context like the one we work with, it is possible to define a graded version of the Nijenhuis torsion for an even vector valued one-form $N$ (see for instance \cite{ILMM}). Its definition is still
$$T(N)(X,Y)=[NX,NY]-N([NX,Y]+[X,NY])+N^2[X,Y]$$
where $X$ and $Y$ are now any two graded vector fields and the brackets are graded Lie brackets of graded vector fields.
\begin{theorem}\label{masterCH}
$$T(N)=0\in CH(\mathcal{T}^{(1,2)}\mathbf{Q},\mathcal{L}_{X})$$
\end{theorem}
\begin{proof}For the proof we need to apply Stokes theorem to the boundary of a new type of moduli space, namely a space of holomorphic curves with one positive and many negative punctures, three of which have special ``coupled'' asymptotic markers, namely the positive one and two of the negative ones. Let us give an idea of the generic element in such moduli space. In the interior of the moduli spaces we have maps from $(\mathbb{P}^1\setminus\{a,b,c,(z_j^-)\},(z_i))$ to $V\times \IR$. The special positive puncture is $a\in\IP^1$, and the two negative ones are $b\in\IP^1$ and $c\in\IP^1$. There are further negative punctures $(z^-_j)$ and marked points $(z_i)$. We have asymptotic markers at all puntcures, as usual, but the markers at $a,b$ and $c$ are mutually constrained (any marker determines the other two) in the following way. Choosing any parametrization of $\IP^1$ and given an asymptotic marker at $a$, we can map it to a marker at $b$ and a marker at $c$ along the only two arcs of circle $l_1, l_2$ issuing from $a$ in the tangent direction of its marker and passing through $b$ or $c$ respectively (notice that this construction is independent of the parametrization, as circles are mapped to circles by M\"obius transformations). We will use these asymptotic markers to define evaluation maps and pull back classes $d\phi_{\gamma_1}$, $d\phi_{\gamma_2}$ and $d\phi_{\gamma_3}$ at the orbits $\gamma_1$, $\gamma_2$ and $\gamma_3$ respectively. This moduli space is introduced because its boundary contains more familiar type of curves. We want to integrate differential forms over such boundary to obtain equations for the corresponding generating functions for moduli spaces we already defined before. Let us then study (what an ideal transversality result should give as) the codimension-one boundary of such moduli spaces. It is clear that every time a curve degenerates into a $2$-floor configurations and one special puncture is separated from the other two, we will have a matching condition for the $S^1$-parametrization at the connecting node/Reeb orbit. As we have seen above, we express this matching by pulling back the diagonal class at the connecting orbit. If the splitting into two floors leaves all of the special punctures on the same floor, then one floor will have a constrained parametrization as we described above, while the other one will be an ordinary unparametrized curve.\\

The following picture, with the usual notations, represents the possible $2$-floor degenerations in such moduli spaces. 

\begin{center}
\includegraphics[width=9.2cm]{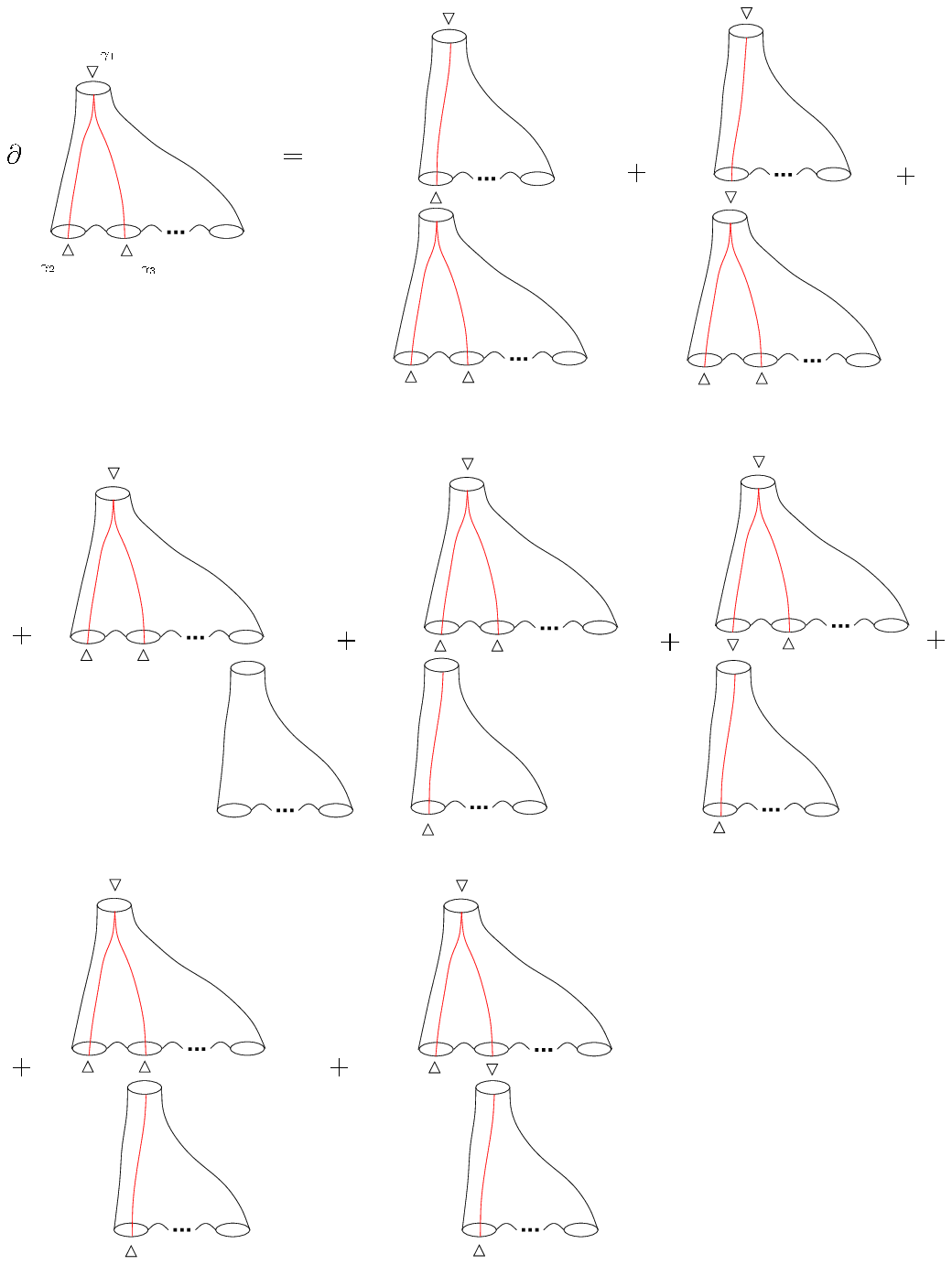}
\end{center}

The two red lines, with common tangent direction at the positive puncture $a\in\IP^1$, represent the image of $l_1$ and $l_2$. They determine coupled evaluation maps to $\gamma_1,\gamma_2$ and $\gamma_3$, along which we can then pull-back $d\phi_{\bar\gamma}$ classes. As usual the triangles point towards the orbit if we pull back a  $d\phi_{\bar\gamma}$ class (i.e. if we decide to constrain the marker) at that orbit. The effect of pulling back the diagonal class at a connecting orbit, as in the proof of the theorem above, is the appearence of two terms in the right hand side of the pictorial equation, whose only difference is the direction of the triangle at the connecting orbit (this happens specifically at the first two terms, or the fourth and fifth, or again at the sixth and seventh).\\

Let us now interpret the right hand side of the equation in terms of generating functions. The second, third, fourth and sixth term form the Lie derivative along $X$ of the $(1,2)$-tensor whose correlator counts the holomorphic curves with three special punctures described above, and hence disappear when taking homology with respect to $\mathcal{L}_X$. The first term spells out as $N^{\gamma_1}_\gamma\left(\frac{\del N^\gamma_{\gamma_2}}{\del q^{\gamma_3}}-\frac{\del N^\gamma_{\gamma_3}}{\del q^\gamma_2}\right)$. Indeed the curves represented in the lower floor of the first summand have two red lines on them, whose tangents at the positive puncture must match but are not otherwise constrained. This matching condition can be expressed (similarly to what happens for the matching condition of two red lines from two different floors at the connecting orbit) by pulling back to the moduli space the form $\frac{1}{\kappa_\gamma}(d\phi_{\bar\gamma}\otimes 1 - 1\otimes d\phi_{\bar\gamma}$)). This form represents the anti-diagonal class in $H^*(\gamma\times\gamma)$, the class of the diagonal in $S^1\times(-S^1)$, where the minus in the second factor comes from the fact that, because of our way of transporting asymptotic markers from $\gamma_2$ and $\gamma_3$ to $\gamma$, the induced maps between the corresponding Reeb orbits have opposite orientations). We therefore get the term $\left(\frac{\del N^\gamma_{\gamma_2}}{\del q^{\gamma_3}}-\frac{\del N^\gamma_{\gamma_3}}{\del q^\gamma_2}\right)$, where this difference corresponds to the difference in the above anti-diagonal in the following way. Each summand of the anti-diagonal form $S^1$-constrains the positive end of one of the two red lines, leaving the other one free. Considering that such lines are both $S^1$-constrained at their negative ends, we end up with a doubly $S^1$-constrained line, while we can just forget the other one, which is $S^1$-constrained just once. Hence, of the three indices involved in the picture (the three punctures with red lines on them), two of them are indices of $N$ (doubly constrained line) and the third is the index of a simple $q$-derivative (a marked negative puncture). We promptly recognize this as the first term of the coordinate expression (\ref{torsion}) of the Nijenhuis torsion of $N$. Finally, the remaining two summands, the fifth and the seventh, give (up to signs corresponding to the grading of $q$-variables) the remaining part of (\ref{torsion}) and we can conclude that, up to $\mathcal{L}_X$-homology, $T(N)=0$.\\
\end{proof}

\begin{corollary}
For any $Y\in CH(\mathcal{T}^{(1,0)}\mathbf{Q},\mathcal{L}_{X})$,
$$\mathcal{L}_Y N=0 \ \in CH(\mathcal{T}^{(1,1)}\mathbf{Q},\mathcal{L}_{X})\hspace{0.5cm} \Rightarrow\hspace{0.5cm}  \mathcal{L}_{N(Y)} N=0 \ \in CH(\mathcal{T}^{(1,1)}\mathbf{Q},\mathcal{L}_{X})$$
\end{corollary}
\begin{proof}
Simply spell out in components the difference $\mathcal{L}_Y N -  \mathcal{L}_{N(Y)} N$ to see that it is proportional to the left-hand side of the master equation for $N$.
\end{proof}

\subsection{Descendant vector fields and $N$-recursion}
The following result shows how the non-equivariant Nijenhuis endomorphism $N$ is related to the geometry of gravitational descendants and the combined knowledge of the primary vector fields $X_{\alpha,0} \in CH(\mathcal{T}^{(1,0)}\mathbf{Q},\mathcal{L}_{X})$ and of the endomorphism $N \in CH(\mathcal{T}^{(1,1)}\mathbf{Q},\mathcal{L}_{X})$ allows for completely recovering all of the descendant vector fields $X_{\alpha,i} \in CH(\mathcal{T}^{(1,0)}\mathbf{Q},\mathcal{L}_{X})$, $i>0$.\\

\begin{theorem}\label{omegarecursionCH}
$$X_{\alpha,n}=N(X_{\alpha,n-1})+C^\mu_{\alpha,n-1} X_{\mu,0}\ \in CH(\mathcal{T}^{(1,0)}\mathbf{Q},\mathcal{L}_{X})$$
where
$$C^\mu_{\alpha,n}=C^\mu_{\alpha,n}(t)=\frac{\del^2}{\del t^\alpha \del t^\nu} \int_V \frac{t^{\wedge (n+3)}}{(n+3)!}\ \eta^{\nu \mu}$$
\end{theorem}
\begin{proof}
Once more we need to study the codimension-$1$ boundary of a moduli space of curves. In this case we consider contact homology curves with three special points: the positive puncture at the orbit $\gamma$, a marked point at which we pull back the unity class $1\in H^*(V)$ and another marked point carrying the $n$-nth descendant of the class $\theta_\alpha\in H^*(V)$ (and no other point carries gravitational descendants). Mapping these three points to $\{0,1,\infty\}\in \IP^1$ we obtain an asymptotic direction at the positive puncture given by the $\IR^+$-line in $\IP^1$ and we constrain such direction as usual via the asymptotic marker at the corrisponding positive Reeb orbit.

\begin{center}
\includegraphics[width=11.5cm]{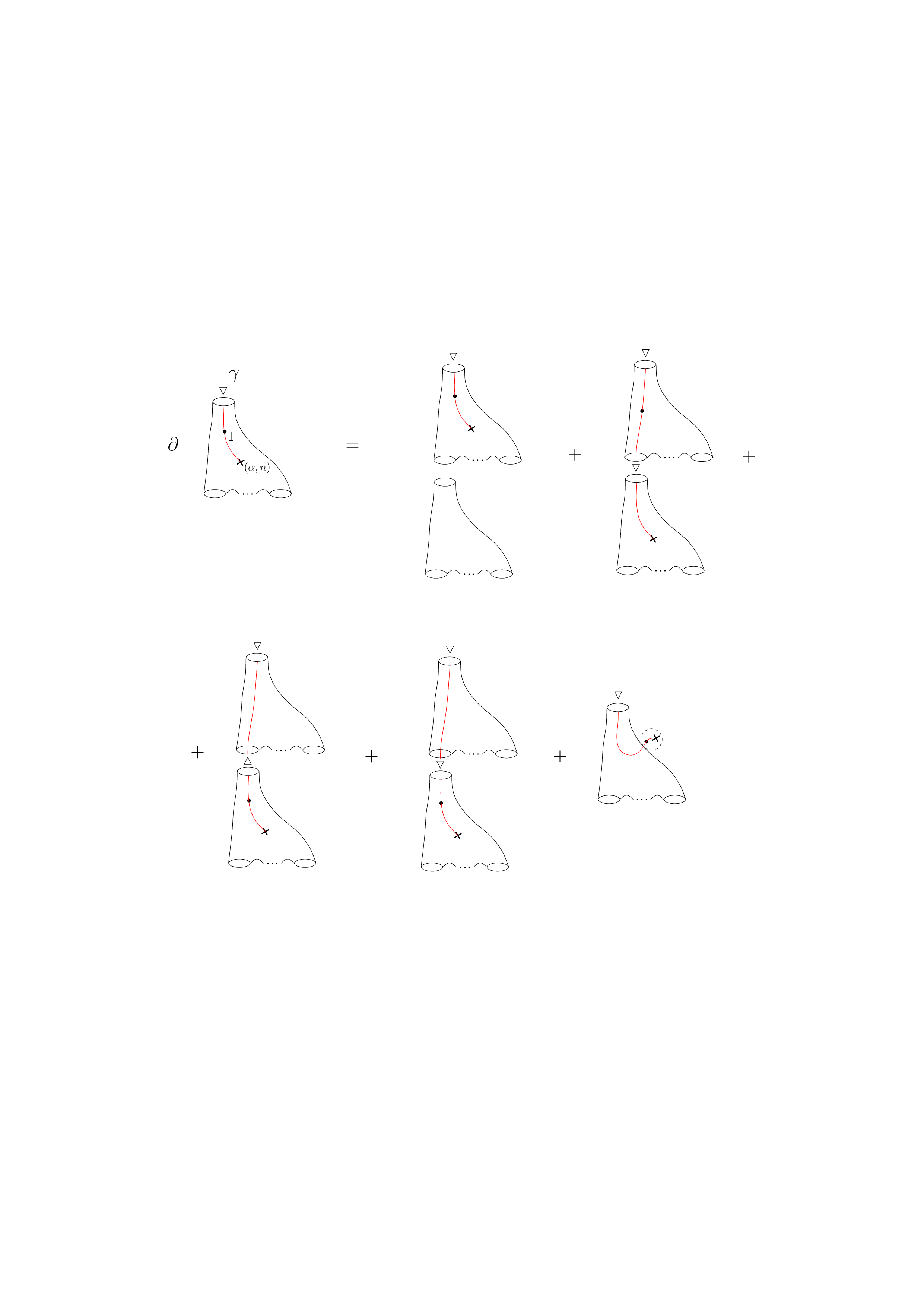}
\end{center}

In the usual way the above picture shows the different types of codimension-$1$ boundary degeneration for such moduli space. We are already familiar with the first four terms of the right-hand side: they represent all possible $2$-floor degenerations of the $1$-floor curve on the right hand side. Notice however that, since the two special marked points are constrained to a line which is in turn $S^1$-constrained at the Reeb orbit, a special kind of codimension-$1$ phenomenon appears, which is not anymore a $2$-floor curve, but is instead a $1$-floor curve with a constant sphere-bubble carrying the two special marked points (which corresponds to the limit where the point carrying the class $1\in H^*(V)$ moves along the $\IR^+$-line to reach the other marked point carrying the descendant), represented as the last term in the right-hand side in the picture. It is easy to convince one-self, by dimension counting of the moduli of each of the two components, that this nodal configuration is a codimension-$1$ phenomenon, but of course, once more, in order to have a rigorous result, the appropriate transversaility and gluing theorems (making this moduli space a well-behaved manifold with corners) are needed.\\

Now we notice that the first and fourth terms in the right-hand side correspond to the Lie derivative along $X$ of a vector field on $\mathbf{Q}$ whose component along $\frac{\del}{\del q^\gamma}$ are given by the correlator counting the curves described above. Notice further that the factor corresponding to top floor in the second summand is zero unless the curve is a constant cylinder, because the marked point carrying the class $1\in H^*(V)$ is always unconstrained along the red line and the only way to achieve a zero-dimensional moduli space is by quotienting out the vertical symmetry in constant cylinders over the Reeb orbit $\gamma$ . Taking homology, what is left can be spelled out as
$$\left( X^{\gamma_1}_{\alpha,n} \delta_{\gamma_1}^\gamma - \frac{\del X^{\gamma_1}_{\alpha,n}}{\del t^1} N_{\gamma_1}^\gamma - \frac{\del C^\mu_{\alpha,n}}{\del t^1} X^\gamma_{\mu,0}\right)\frac{\del}{\del q^\gamma}=0\ \in CH(\mathcal{T}^{(1,0)}\mathbf{Q},\mathcal{L}_{X}) $$
where $C^\mu_{\alpha,n}$ is the term accounting for the constant bubbles with one psi-class to the power $n$. Such term is easily calculated from the well known fact that, on the Deligne-Mumford space of genus $0$ curves with $r$ marked points,
$$\int_{\CM_{0,r}} \psi_i^n = \left\{\begin{array}{c}1,\ r=n+3\\0,\ r\neq n+3\end{array}\right. $$
Finally we need to use the string equation of Theorem \ref{SDD} on the above equation to obtain the statement.
\end{proof}

\begin{corollary}\label{corollaryCH}
$$ X_{\alpha,n}= \sum_{k=0}^n C^\mu_{\alpha,n-k-1}\ N^k(X_{\mu,0}) \ \in CH(\mathcal{T}^{(1,0)}\mathbf{Q},\mathcal{L}_{X}) $$
where
$$C^\mu_{\alpha,n}=C^\mu_{\alpha,n}(t)=\frac{\del^2}{\del t^\alpha \del t^\nu} \int_V \frac{t^{\wedge (n+3)}}{(n+3)!}\ \eta^{\nu \mu}$$
\end{corollary}
\begin{proof}
Just apply Theorem \ref{omegarecursionCH} $n$ times to $X_{\alpha,n}$.
\end{proof}
Naturally the above theorem and corollary hold for any choice of auxiliary data given the completely covariant behaviour of the equations. In particular the above corollary shows how the descendant vector fields $X_{\alpha,n}$ are expressed in closed form in terms of the primary vector fields $X_{\alpha,0}$ and the endomorphism $N$.
\begin{example}
Consider again the example of $V=S^1$. In this case we have $\bar{X}:=X|_{\tau^1=0}=0$ and $\bar{X}_{1,n}:= \left. \frac{\del X}{\del \tau^{1,n}}\right|_{\tau=0}$, with $\bar{X}^k_{1,0}=k q^k$, where as before we write $k$ for the index $k\gamma$ associated to the $k$-th multiple of the orbit $\gamma=V$.\\

Applying the above Theorem \ref{omegarecursionCH} we obtain
$$\bar{X}^l_{1,1}=\bar{X}^k_{1,0} \bar{N}_k^l + \bar{C}^1_{1,0} \bar{X}^l_{1,0}$$.
Here $\bar{C}^1_{1,n}=C^1_{1,n}|_{\tau^1=0}=\frac{(t^1)^{n+1}}{(n+1)!}$, hence we obtain
\begin{equation*}
\begin{split}
\bar{X}^l_{1,1}&=l\ t^1 q^l + \sum_{0<k<l} (l-k) q^k q^{l-k}=\\
&=l\ t^1 q^l + \frac{1}{2} \left(\sum_{0<k<l} (l-k) q^k q^{l-k} + \sum_{0<k'<l} k' q^{-k'+l}q^{k'}\right)= \\
&=l\ t^1 q^l + \frac{l}{2} \sum_{0<k<l} q^k q^{l-k}
\end{split}
\end{equation*}
and, with the same procedure we obtain, by Corollary \ref{corollaryCH} and denoting $q^0:=t^1$,
$$\bar{X}^l_{1,n}=\frac{l}{(n-1)!} \sum_{\substack{k_1,\ldots,k_n\geq0\\k_1+\ldots+k_n=l}} q^{k_1} \ldots q^{k_n}$$
(here one needs to use the following trick
\begin{equation*}
\begin{split}
\sum_{\substack{k_1,\ldots,k_n\geq0\\k_1+\ldots+k_n\leq l}} &(l-k_1-\ldots-k_n)\  q^{k_1}\ldots q^{k_n} q^{l-k_1-\ldots-k_n}=\\
=\frac{1}{n} &\left(\sum_{\substack{k'_1,k_2,\ldots,k_n\geq0\\k'_1+k_2+\ldots+k_n\leq l}} k'_1 q^{l-k'_1-k_2-\ldots-k_n}+\right.\\
&\hspace{0.5cm}+\ldots+\left. \sum_{\substack{k_1,\ldots,k_{n-1},k'_n\geq0\\k_1+\ldots+k_{n-1}+k'_n\leq l}} k'_n q^{l-k_1-\ldots-k_{n-1}-k'_n} \right)=\\
\frac{l}{n} &\sum_{\substack{k_1,\ldots,k_n\geq0\\k_1+\ldots+k_n\leq l}}   q^{k_1}\ldots q^{k_n} q^{l-k_1-\ldots-k_n}
\end{split}
\end{equation*}
to take the numerical coefficient out of the sum).
\end{example}

\section{$\omega$-recursion in rational SFT}

An approach similar to the one we used above for contact homology should also work in the case of full rational SFT, the main difference coming from the presence of non-constant nodal curves which is very naturally incorporated in the algebraic formalism of Lie derivatives and tensor fields by trading the Nijenhuis operator $N$ for a bivector $\omega$ well defined on SFT homology. This section should be seen as a research announcements of the algebraic results that follow from assuming that the anlysis of the previous sections carries over without troubles to the more general context of SFT of contact manifolds (or even more general stable Hamiltonian structures). In particular this means that we are introducing new moduli spaces of $S^1$-parametrized rational curves and we are assuming the same type of degenerations of such curves that we have studied above, also serve as compactification (with the appropriate transversality and gluing theorems) in this case (with the substantial difference that nodal degenerations inside a given floor separating the top and bottom puncture of the given $S^1$-parametrized cylinder happen in codimension $1$, as explained below).

\subsection{The $\omega$ bivector in rational SFT}
For a target contact manifold $V$ and compatible cylindrical almost complex structure $J$ on $V\times \IR$, consider the following moduli spaces of punctured $S^1$-parametrized cylinders with marked points. We start with the fully parametrized space  $$\IM^{S^1,0}_{r,A}((\gamma_0,\pm),(\gamma_\infty,\pm),\Gamma^+,\Gamma^-))$$ consisting of tuples $(u,(z_k^{\pm}),(z_i))$, where $(z^{+}_k)$,$(z^{-}_j)$,$(z_i)$ are three disjoint ordered sets of points on $\IP^1 \setminus (\{0,\infty\})=S^1\times\IR$ (positive and negative punctures, and $r$ additional marked points). The map $u: \Si \to \IR \times V$  from the punctured Riemann surface $\Si = \IP^1 \setminus (\{0,\infty\} \cup  \{(z_k^{+})\} \cup \{(z_k^{+})\})$ is required to satisfy the Cauchy-Riemann equation
\begin{equation*}
 \CR_{\Ju} u = du + \Ju(u) \cdot du \cdot i = 0
\end{equation*}
with respect to the complex structure $i$ on $\IP^1$. Assuming we have chosen cylindrical coordinates $\psi^{\pm}_k: \IR^{\pm}
\times S^1 \to \Si$ around each puncture $z^{\pm}_k$, in the sense that $\psi_k^{\pm}(\pm\infty,t)=z_k^{\pm}$,
the map $u$ is additionally required to show for all $k=1,...,n^{\pm}$ the
asymptotic behaviour
\begin{equation*}
 \lim_{s\to\pm\infty} (u \circ \psi^{\pm}_k) (s,t+t_0) =
 (\pm \infty,\gamma^{\pm}_k(T^{\pm}_kt))
\end{equation*}
with $t_0\in S^1$ and with orbits $\gamma^{\pm}_k\in \Gamma^{\pm}$, where $T^{\pm}_k>0$ denotes period of $\gamma^{\pm}_k$, and analogous asymptotic behaviour at $0$ and $\infty$ for $s\to\pm\infty$ (the signs here correpond to the signs in $(\gamma_0,\pm),(\gamma_\infty,\pm)$ in the notation for the moduli space) for orbits $\gamma_0$ and $\gamma_\infty$ with respect to the natural coordinates on $S^1\times\IR$. We assign to each curve an absolute homology class $A$ employing as usual a choice of spanning surfaces. In order to obtain the $S^1$-parametrized space $\IM^{S^1}_{r,A}((\gamma_0,\pm),(\gamma_\infty,\pm),\Gamma^+,\Gamma^-)$ we only divide out the $\IR$-component of the $S^1\times \IR$ group of automorphisms of the cylinder $\IP^1 \setminus (\{0,\infty\}$ and, as usual, the $\IR$-action coming from the cylindrical target $V\times\IR$ as well.\\

The compactification $\CM^{S^1}=\CM^{S^1}_{r,A}((\gamma_0,\pm),(\gamma_\infty,\pm),\Gamma^+,\Gamma^-)$ is obtained as usual by adding multifloor $S^1$-parametrized curves. In genus zero each floor has only one non-trivial connected component, all the others being trivial cylinders over Reeb orbits (possibly $S^1$-parametrized). If the $0$ and $\infty$ punctures determining the $S^1$-parametrization appear on the $k$-th and $l$-th floor of a $n$-floor curve, it means that all non-trivial curves appearing on the $m$-th floor are $S^1$-parametrized curves when $k\leq m\leq l$ and ordinary unparametrized curves when $m<k$ or $m>l$. As anticipated above, nodal curves should also be added to the picture, both in the usual codimension $\geq 2$ strata (when the node does not separate the $0$- and $\infty$-puncture, and as a new type of codimension $\geq1$ stratum, when a node separates the $0$- and $\infty$-punctures on a given floor. In the second case each component carries its own $S^1$-parametrization with respect to the node and the $0$- or $\infty$-puncture respectively and such two parametrizations have no matching condition at the common node (it is easy to convince one-self that this type of nodal degeneration of $S^1$-parametrized cylinders should happen in codimension $1$, just by counting the dimensions of the moduli of each of the two components).\\

As in the contact homology case, the space $\CM^{S^1}$ carries, besides the usual evaluation maps at marked points, orbits and asymptotic markers, extra evaluation maps at the punctures at $0$ and $\infty$ to the corresponding target simple Reeb orbits given by the special $S^1$-coordinate on the curve,
$$\ev_{\pm\infty,0}:\CM^{S^1}\to\bar\gamma_0 \simeq S^1$$
$$\ev_{\pm\infty,\infty}:\CM^{S^1}\to  \bar\gamma_\infty \simeq S^1.$$\\

We form the space $\CM^{S^1}_{r,n^+,n^-,A}((\gamma_0,\pm),(\gamma_\infty,\pm))$ by taking the union over $\Gamma^+$ and $\Gamma^-$ of all the spaces $\CM^{S^1}_{r,A}((\gamma_0,\pm),(\gamma_\infty,\pm),\Gamma^+,\Gamma^-)$ with $|\Gamma^-|=n^+$ and $|\Gamma^-|=n^-$.\\

We will now define a $(1,1)$-tensor on the Poisson super-space $\mathbf{V}_0$ underlying the Poisson subalgebra $\PP_0\subset \PP$ generated by $p$ and $q$-variables and even $t$-variables only. In other words, if $H^*(V)=H^*_{\text{even}}(V) \oplus H^*_{\text{odd}}(V)$, with $H^*_{\text{even}}(V) =<\theta_1,\ldots,\theta_N>$ and $H^*_{\text{even}}(V) =<\Theta_1,\ldots,\Theta_L>$, we denote by $t^\alpha$ the (even) formal variable associated to the class $\theta_\alpha$, $\alpha=1,\ldots,N$, and by $\tau^{\check{\alpha}}$ the (odd) formal variable associated to $\Theta_{\check{\alpha}}$, $\check{\alpha}=1,\ldots,L$. Then $\PP_0=\PP|_{\tau=0}$.  Correspondingly we define $\Ih^0:=\Ih|_{\tau=0}$ and $\Ih_{\check{\alpha},n}:=\frac{\del \tilde{\Ih}}{\del \tau^{\check{\alpha},n}}|_{\tau=0}$. Notice that the Hamiltonians $\Ih_{\check{\alpha},n}$ are always even elements in $\PP_0$.\\

We will denote globally by $v^A$ any of the coordinates $t^\alpha$, $p^{a}$ or $q^{a}$ (again, to avoid confusion, we have raised the indices of $p$ and $q$ variables, coherently with their interpretation as coordinates for $\mathbf{V}_0$). We will always use lower case roman indices (e.g. $v^a$) to refer indistinctly to a $p$ or $q$ variable, greek indices (e.g. $v^\alpha$) for $t$ variables and checked greek indices (e.g. $v^{\check{\alpha}}$) for $\tau$-variables. Also, for convenience, we let $v^{(\gamma,+)}:=p^\gamma$ and $v^{(\gamma,-)}:=q^\gamma$, so that the roman upper case indices $A$, $B$, etc. can take the values $\alpha$, $\beta$, etc. when the corresponding variable is a $t$-variable, or the values $(\gamma_1,\pm)$, $(\gamma_2,\pm)$, etc. or again simply $a$ and $b$ when the corresponding variable is a $p$ or $q$-variable. Notice also that, as opposed to what we did for the space $\mathbf{Q}$, the $t$-variables are treated here as genuine coordinates and not as parameters. In particular, the Poincar\'e metric $\eta$ will split into two blocks (one the transpose of the other) always pairing even with odd cohomology classes. We denote the matrix corresponding to each of such blocks by $\eta_{\alpha \check{\alpha}}=\eta(\theta_\alpha,\Theta_{\check{\alpha}})$.\\

Using such coordinates we define the (graded) symmetric bivector
$$\omega =\ \omega^{AB} \ \frac{\del}{\del v^A} \otimes \frac{\del}{\del v^B}  $$
where we sum over repeated indices, with
\begin{equation*}
\begin{split}
\omega^{(\gamma_1,\pm) (\gamma_2,\pm)} =& \sum  \frac{1}{r! n^- !} \int_{\CM^{S^1}_{r,n^+,n^-,A}((\gamma_1,\pm),(\gamma_2,\pm))} \bigwedge_{i=1}^r \ev_i^* t\
\bigwedge_{j=1}^{n^+}(ev^*_{+,j}p \wedge \ev^*_{+\infty,j}d\phi_{\bar\gamma^+_j})\\
&\bigwedge_{j=1}^{n^-}(ev^*_{-,j}q \wedge \ev^*_{-\infty,j}d\phi_{\bar\gamma^-_j})\
\wedge (\ev_{\pm\infty,0}^* d\phi_{\bar\gamma_1}) \wedge (\ev_{\mp\infty,\infty}^* d\phi_{\bar\gamma_2})
\end{split}
\end{equation*}
and
\begin{equation}\label{omegapart2}
\omega^{a \alpha} = \left. \Pi^{ab} \eta^{ \alpha \check{\mu}}\frac{\del \Ih_{\check{\mu},0}}{\del v^b}\right|_{\tau=0} = \omega^{\alpha a}
\end{equation}
and zero otherwise.\\

As in contact homology, from the index formula for the virtual dimension of the moduli space of SFT curves, we have
$$|\omega|=-2.$$

Algebraically, the SFT differential $d_0:\PP_0\to\PP_0$, defined as the vector field $X^0=X_{\Ih^0}=\{\Ih^0,\cdot\}:\PP_0\to\PP_0$ induces a differential $\mathcal{L}_{X^0}$ on the space of $(k,l)$-tensor fields $\mathcal{T}^{(k,l)}\mathbf{V}_0$ on the Poisson super-space $\mathbf{V}_0$ underlying $\PP_0$. The resulting homology, which we denote by $H_*(\mathcal{T}^{(k,l)}\mathbf{V}_0;\mathcal{L}_{X^0})$, is a module over $H_*(\PP_0,d_0)=H_*(\mathcal{T}^{(0,0)}\mathbf{V}_0;\mathcal{L}_{X^0})$ and is an invariant of the contact structure on $V$. In particular, for two different choices of form $\lambda^{\pm}$, cylindrical almost complex structure
$\Ju^{\pm}$ , representatives for the classes $[\theta_\alpha],[\Theta_{\check{\alpha}}]\in H^*(V)$ and $[d\phi_{\bar\gamma}]\in H^*(S^1)$, abstract polyfold perturbations and sequences of coherent collections of sections $(s^{\pm}_j)$, there exist an isomorphism
$$d\varphi^\pm: H(\mathcal{T}^{(k,l)}\mathbf{V}_0^+,\mathcal{L}_{X^{0+}}) \to H(\mathcal{T}^{(k,l)}\mathbf{V}_0^-,\mathcal{L}_{X^{0-}})$$
which is simply the lift to the tensor algebra of the isomorphism
$$\varphi^\pm: H_*(\PP_0^+;d_0^+) \to H_*(\PP_0^-;d_0^-),$$
constructed in \cite{EGH} by studying curves in the cobordims $W=\overrightarrow{V^+V^-}$ interpolating between the two different choices (see also the discussion on invariance for satellites there).\\

Moreover the descendant hamiltonians $h_{\check{\alpha},n}$ induce covariant (with respect to $d\varphi^\pm$) Hamiltonian vector fields $X_{\check{\alpha},n} \in H_*(\mathcal{T}^{(1,0)}\mathbf{V}_0;\mathcal{L}_{X^0})$, $\check{\alpha}=1,\ldots,L$, $n=0,1,2,\ldots$.\\

\begin{theorem}
$$\mathcal{L}_{X^0}\omega = 0$$
\end{theorem}
\begin{proof}
We proceed exactly as in the contact homology case, only keeping in mind that, this time, nodal configurations can appear in codimension $1$ when studying the moduli spaces, relevant for $N$, of curves with a doubly $S^1$-constrained line joining the two special $0$ and $\infty$ punctures. Indeed, for such extra boundary, containing nodal curves where the node separates the $0$ and $\infty$ puncture on the same level, the matching condition translates into a gluing condition for the domains at the node. This breaks the usual $S^1$-symmetry, forcing this phenomen to occur in codimension one (as a simple dimension check for the involved moduli spaces will show). Because of our definition of the $t$-components of $\omega$, equation (\ref{omegapart2}), this can be expressed as the term $ \omega^{A \mu} \frac{\del (X^0)^B}{\del t^\mu} \frac{\del}{\del v^A} \otimes \frac{\del}{\del v^B}$ and $ \frac{\del (X^0)^A}{\del t^\mu} \omega^{\mu B}  \frac{\del}{\del v^A} \otimes \frac{\del}{\del v^B}$  in the Lie derivative $\mathcal{L}_{X^0} \omega$, coherently with the fact that, in the full SFT picture, our formal Poisson manifold $V_0$ has coordinates $t^\alpha$, beside $p^a$ and $q^a$.
\end{proof}

\vspace{0.5cm}

\subsection{Descendant Hamiltonian vector fields and  $\omega$-recursion}

The following result is the analogue of Theorem \ref{omegarecursionCH} (and proved in completely similar way) for the rational SFT case, and shows how the non-equivariant bivector $\omega$ is related to the geometry of gravitational descendants and the combined knowledge of differential of the Hamiltonian $d\Ih_{\check{\alpha},n} \in H_*(\mathcal{T}^{(0,1)}\mathbf{V}_0,\mathcal{L}_{X^0})$ and of the graded symmetric bivector $\omega \in H_*(\mathcal{T}^{(2,0)}\mathbf{V}_0,\mathcal{L}_{X^0})$ allows to recover the descendant vector fields $X_{\alpha,n+1} \in H_*(\mathcal{T}^{(1,0)}\mathbf{V}_0,\mathcal{L}_{X^0})$, $n\geq0$. Notice however how, in general, this is not equivalent to recovering the Hamiltonians $h_{\alpha,n+1}$ themselves.\\


\begin{theorem}\label{omegarecursionSFT}
$$X_{\check{\alpha},n+1}=\Pi(\cdot,d \Ih_{\check{\alpha},n+1})=\omega(\cdot,d\Ih_{\check{\alpha},n})\qquad \in H_*(\mathcal{T}^{(1,0)}\mathbf{V}_0,\mathcal{L}_{X^0}) $$
\end{theorem}
\begin{proof}
The statement is proved precisely in the same way as for Theorem \ref{omegarecursionCH}. Notice only that the analogue of the term containing the constants $C^\mu_{\alpha,k}$, counting nodal curves, in this case is absorbed in the Lie derivative that vanishes in homology.
\end{proof}
\vspace{0.5cm}
Notice that the above recursion makes sense for $n=-1$ too if we define $\Ih_{\check{\alpha},-1}:=\eta_{\check{\alpha}\beta}t^\beta$. Then all of our sequences of Hamiltonians $\Ih_{\check{\alpha},n}$ satisfy to a recursion which starts from a Casimir at level $n=-1$. This allows to deduce commutativity $\{\Ih_{\check{\alpha},i},\Ih_{\check{\beta},j}\}=0$, which we know to hold on homology, simply from the recursion, since
\begin{equation*}
\begin{split}
\{\Ih_{\check{\alpha},i},\Ih_{\check{\beta},j}\}&=\omega(d\Ih_{\check{\alpha},i},d\Ih_{\check{\beta},j-1})=\\ &=-\{\Ih_{\check{\alpha},i+1},\Ih_{\check{\beta},j-1}\}=\\
&=\ldots=\\
&=(-1)^{j+1}\{\Ih_{\check{\alpha},i+j+1},\Ih_{\check{\beta},-1}\}=0.
\end{split}
\end{equation*}
\vspace{0.5cm}

\begin{example}\label{circle}
Consider again the case $V=S^1$ with $t=t^1\theta_1 + \tau^1 \Theta_1$, $\theta_1=1$ and $\Theta_1=d\varphi$ where $\varphi$ is the angular coordinate on $S^1$. Here, as in any other circle bundle over a symplectic manifold with even cohomology, $\Ih^0=0$ and everything happens at chian level. Even in the full rational SFT case, it is straightforward to compute $\omega$. We write $\pm k$ for the index $(k\gamma,\pm)$ associated to the $k$-th multiple of the positive or negative orbit $\gamma=V$ and we use the index $0$ to refer to the component along $t^1$ (or, in other words, $v^0=t^1$). From the dimension formula for the moduli
space of SFT-curves we see that the only nonzero components of $\omega$ correspond to
branched covers of the target $\mathbb{P}^1 \setminus \{0,\infty\}$ by an $S^1$-parametrized cylinder with an
extra puncture and another non-marked branch point (whose target $S^1$-coordinate
is fixed by constraining at both punctures the chosen $S^1$-parametrization of the
source cylinder). This way we immediately see that
$$\omega^{kl}=(k+l)v^{k+l},\hspace{1cm} k,l\in\mathbb{Z}$$
Applying $\omega$-recursion we can recover the $n$-th descendant Hamiltonian. Indeed, let us start with
$$\Ih_{1,0}= \frac{1}{2}\sum_k v^{-k}v^k.$$
Recursion tells us
\begin{equation*}
\begin{split}
\frac{\del \Ih_{1,1}}{\del v^j} \Pi^{jl} &= l \frac{\del \Ih_{1,1}}{\del v^{-l}}=  \sum_{k} (k+l) v^{-k} v^{k+l} = \\
&= \frac{1}{2} \left(\sum_k (k+l) v^{-k} v^{k+l} + \sum_{k'} (-k') v^{k'+l} v^{-k'}\right)=\\
&= \frac{l}{2} \sum_k v^{-k} v^{k+l}
\end{split}
\end{equation*}
from which we deduce
$$\Ih_{1,1}=\frac{1}{6} \sum v^{-k} v^{k+l} v^{-l}.$$
Notice that, actually, for $l=0$, the above equation is void, as we expected. The same procedure can be reiterated to find
$$\Ih_{1,n}= \frac{1}{n!} \sum_{k_1+\ldots+k_n=0} v^{k_1}\ldots v^{k_n}.$$
\end{example}

\begin{example}
It is actually possible to explicitly compute the operator $\omega$ for the stable Hamiltonian structure of the type described in example \ref{S1bundle}, where $V$ is the trivial $S^1$-bundle over a symplectic manifold $(M,\omega_M)$. The rational Symplectic Field Theory of such manifold $V=S^1\times M$ requires a Morse-Bott approach (as the Reeb orbits come in a family parametrized by $M$, every fiber $S^1$ being one such orbit) and is described in \cite{B} and \cite{EGH} for the case of contact manifolds. The trivial bundle case can be treated analogously and, in case the base symplectic manifold is K\"ahler, it even falls inside the relative Gromov-Witten theory approach (together with all other holomorphic $S^1$-bundles) as described in \cite{K}. For simplicity we will assume $M$ to have only even cohomology, $H^{\text{odd}}(M)=0$. Of course one has $H^*(V)=H^*(M)\oplus (H^*(M)\otimes d\varphi)$ where $\varphi$ is the fiber coordinate, while $H_2(M)=H_2(V)$. We then choose a basis $\Delta_1,\ldots,\Delta_N$ of $H^*(M)$ and denote by $\eta^{\alpha \beta}$ the Poincar\'e pairing on $H^*(M)$. We pull back $\Delta_1,\ldots,\Delta_N$ to $H^*(V)$ and complete them to a basis by adding odd classes $\Theta_1,\ldots,\Theta_N$, with $\Theta_k=\pi^*(\Delta_k)\otimes d\varphi$. By setting $t^\alpha=v^{\alpha,0}$, $\alpha=1,\ldots,N$ and using the unified notation $v^{\alpha,k}$, $\alpha=1,\ldots,N$, and either $k<0$ or $k>0$ for $p$ and $q$ variables associated (to cohomology classes of) the space $M$ of Reeb orbits, we can define generating functions
$$v^\alpha(x):=\sum_{k\in \mathbb{Z}} v^{\alpha,k} e^{i k x}, \ \alpha=1,\ldots,N$$
Let us denote by $\mathbf{f}_M=\mathbf{f}_M(t)$ the full descendant rational Gromov-Witten potential of $M$, where $t$ is short-hand notation for $t^{\alpha,n}$, $\alpha=1,\ldots,N$, $n=0,1,2,\dots$, the formal variables associated in Gromov-Witten theory to the elements in our basis for $H^*(M)$ and thier descendants. Let $f=f(v)=\mathbf{f}_M|_{v^{\alpha,n}=0,\ n>0}$ be the primary potential and $h_{\alpha,n}=\left.\frac{\del \mathbf{f}_M}{\del v^{\alpha,n}}\right|_{v^{\beta,j}=0,\ j>0}$ be the one-descendant components (often called $J$-function).\\

It is a result of Bourgeois \cite{B} that can be found also in \cite{EGH} that one can write the SFT-hamiltonians in terms of the GW-potential in the following way:
$$\Ih_{\check{\beta},n}=\frac{1}{2 \pi} \int_0^{2 \pi} h_{\beta,n}(t=v(x))dx.$$
Moreover notice that $\Ih^0=0$ because of the $S^1$-symmetry of the target.\\

We use now topological recursion relations (see e.g. \cite{G}) for the rational GW theory of $M$ to recover the explicit form of the operator $\omega$ for an $S^1$-bundle. Indeed, because of the form of the Poisson tensor, which can be written in terms of the generating functions as the formal distribution (see \cite{R2})
$$\{v^\alpha(x),v^\beta(y)\}=-i\delta ' (x-y)$$
we can write the $(\alpha,k)$-component of the hamiltonian vector field relative to $\Ih_{\check{\beta},n}$ as
\begin{equation*}
\begin{split}
X^{\alpha,k}_{\check{\beta},n}=\frac{dv^{\alpha,k}}{dt^{\check{\beta},n}}&=  -\frac{i}{2 \pi} \eta^{\alpha\mu} \int_0^{2 \pi}\left( \frac{d}{d x}\frac{\del h_{\beta,n}}{\del t^\mu} \right)e^{i k x}dx\\
&=-\frac{i}{2 \pi} \eta^{\alpha\mu} \int_0^{2 \pi}\frac{\del^2 h_{\beta,n}}{\del t^\mu \del t^\nu} v^\nu_x e^{i k x}dx
\end{split}
\end{equation*}
Using topological recursion relations for the rational Gromov-Witten theory of $M$
$$\frac{\del h_{\beta,n}}{\del t^\mu \del t^\nu}=\frac{\del h_{\beta,n-1}}{\del t^\epsilon} \eta^{\epsilon\delta} \frac{\del^3 \mathbf{f}}{\del t^\delta \del t^\mu \del t^\nu}$$
we get
\begin{equation*}
\begin{split}
X^{\alpha,k}_{\check{\beta},n}&=-\frac{i}{2 \pi} \eta^{\alpha\mu} \int_0^{2 \pi} \frac{\del^3 \mathbf{f}}{\del t^\mu \del t^\nu \del t^\delta} v^\nu_x \eta^{\delta\epsilon} e^{ikx}\frac{\del h_{\beta,n-1}}{\del t^\epsilon} dx\\
&=\sum_{l\in\mathbb{Z}}\left(-\frac{i}{2 \pi} \eta^{\alpha\mu} \int_0^{2 \pi} \frac{\del^3 \mathbf{f}}{\del t^\mu \del t^\nu \del t^\delta} v^\nu_x e^{i(k+l)x}dx \ \eta^{\delta\epsilon}\right)\left( \frac{1}{2 \pi}\int_0^{2 \pi} \frac{\del h_{\beta,n-1}}{\del t^\epsilon} e^{-ilx} dx\right)\\
&=\sum_{l\in\mathbb{Z}}\left(-\frac{i}{2 \pi} \eta^{\alpha\mu} \int_0^{2 \pi} \frac{\del^3 \mathbf{f}}{\del t^\mu \del t^\nu \del t^\delta} v^\nu_x e^{i(k+l)x}dx \ \eta^{\delta\epsilon}\right) \frac{\del \Ih_{\check{\beta},n-1}}{\del v^{\epsilon,l}}
\end{split}
\end{equation*}
from which we read the expression for the components of the bivector $\omega$
\begin{equation*}
\begin{split}
\omega^{(\alpha,k)(\epsilon,l)}&=-\frac{i}{2 \pi} \eta^{\alpha\mu} \int_0^{2 \pi} \frac{\del^3 \mathbf{f}}{\del t^\mu \del t^\nu \del t^\delta} v^\nu_x e^{i(k+l)x}dx \ \eta^{\delta\epsilon}=\\
&=-\frac{i}{2 \pi} \int_0^{2 \pi} c^{\alpha \epsilon}_\nu(t=v(x))\  v^\nu_x\  e^{i(k+l)x}dx 
\end{split}
\end{equation*}
where $ c^{\alpha \epsilon}_\nu= c^{\alpha \epsilon}_\nu(t^1,\ldots,t^N)$ are the structure functions of the quantum product on the cotangent bundle of quantum cohomology of $M$. In the formal loop space formalism (see e.g. \cite{DZ}) this last formula reads
$$\omega(\delta v^\alpha(x),\delta v^\epsilon(y))=c^{\alpha \epsilon}_\nu(v(x))\  v^\nu_x\ \delta(x-y)$$

Notice in particular how this formula reduces to what we computed in example \ref{circle} when $M=\text{pt}$ and $V=S^1$.
\end{example}

\end{document}